\documentclass[12pt]{article}
\usepackage{amsmath,amsfonts,amssymb,amsthm,amscd}
\title{ \sc  on classification of continuous first order theories}
\date{\today}
\author{Karim Khanaki\thanks{Partially supported by IPM grant 1403030025}\\Arak University of Technology, and \\Institute for Research in Fundamental Sciences  (IPM)}

\newtheorem{Theorem}{Theorem}[section]
\newtheorem{Convention}[Theorem]{Convention}
\newtheorem{Proposition}[Theorem]{Proposition}
\newtheorem{Definition}[Theorem]{Definition}
\newtheorem{Remark}[Theorem]{Remark}
\newtheorem{Lemma}[Theorem]{Lemma}
\newtheorem{Corollary}[Theorem]{Corollary}
\newtheorem{Fact}[Theorem]{Fact}

\def\dotminus{\mathbin{\ooalign{\hss\raise1ex\hbox{.}\hss\cr
  \mathsurround=0pt$-$}}}

\begin{document}
\maketitle

\begin{abstract} 
We present several new characterizations of $IP$ (the independence property) and $SOP$ (the strict order property) for continuous first-order logic and explore their connections to functional analysis and Banach space theory. Furthermore, we propose new dividing lines for unstable theories by examining subclasses of Baire-1 functions. We also explain why one should not expect a perfect analogue of Shelah's theorem—namely, that a theory is unstable if and only if it has $IP$ or $SOP$—to hold for real-valued logics, particularly in the context of continuous logic.
\end{abstract}

\section{Introduction} \label{1}

In 1971, Saharon Shelah proved the following pivotal theorem in \cite{Sh}:

\medskip\noindent
Shelah's Theorem\footnote{In this article, any reference to Shelah's theorem refers to this result (for classical logic).}:
(\cite[Theorem~II.4.7]{Shelah})
A complete first-order theory has the order property ($OP$) if and only if it has the independence property ($IP$) or the strict order property ($SOP$).

\medskip\noindent
This theorem is of great importance for  Shelah's classification of (classical) first-order theories and serves as the focal point of both \cite{K-Baire} and the present paper. In some articles, it is claimed that a straightforward translation of the proof of Shelah's theorem for classical first-order logic could provide a proof of this theorem for continuous logic as developed in  \cite{BBHU}.  
  We conjecture that this claim is false, and in Section 5, we outline the strategy for constructing a counterexample. We argue\footnote{See Section~\ref{discussion} for further discussion.}  that finding such a counterexample is related to a specific class of Banach spaces, namely, nonreflexive Banach spaces that do not contain $c_0$ or $\ell_1$.
However, what is more important to us than disproving this claim is the classification of continuous first-order theories and the characterization of model-theoretic properties in terms of functional analysis and Banach space theory. This classification not only enhances our understanding of continuous logic but also sheds light on classical logic, offering fundamental insights. We believe that the absence of a ``perfect analogue of Shelah's theorem in continuous logic" is not solely a limitation but also has {\em positive} aspects and contributions.

In \cite{K-Baire}, we studied the relationship between $OP$, $IP$, and $SOP$ in classical logic and subclasses of Baire-1 functions, demonstrating the correspondence between Shelah's theorem and the Eberlein--\v{S}mulian Theorem (cf. Fact~\ref{ES theorem} below). The present paper extends the results of \cite{K-Baire} by introducing a new approach to Shelah's classification theory in the context of continuous first-order logic \cite{BBHU}. 

This approach is grounded in the observation that studying the model-theoretic properties of formulas in {\em models}—rather than focusing solely on these properties in {\em theories}—yields a more refined stability theory and establishes significant connections between model theory and other areas of mathematics. 

We aim to address the question of whether a version of Shelah's theorem exists for *continuous logic* by investigating the precise placement of continuous model-theoretic properties within the framework of function spaces. However, while answering this question is one of our objectives, understanding the relationship between model theory and function spaces remains our primary goal, with the former being a natural consequence of the latter.


Continuous logic \cite{BBHU, BU} is a generalization of standard first-order logic designed to study analytic structures such as metric spaces, measure algebras, and Banach spaces, including Banach lattices and $C^*$-algebras, among others. While there are several other alternatives to standard first-order logic 
 for studying metric structures, continuous logic offers many advantages over these alternatives.

Although this logic is particularly suited to structures arising from functional analysis, it is {\em almost}\footnote{See Section~\ref{discussion} for some non-parallelisms.} strikingly parallel to standard first-order logic. Moreover, it retains many desirable characteristics of first-order model theory, such as the compactness theorem, the L\"{o}wenheim--Skolem theorems, the omitting types theorem, and fundamental results from stability theory. 
As we will demonstrate in this paper, continuous logic provides a natural framework for classification theory in the sense of \cite{Shelah}.

Surprisingly, the study of desirable properties within a model and for the extensions or theory of a model in Banach space theory has already been explored by Bessaga--Pe\l czy\'{n}ski \cite{BP58}, Odell--Rosenthal \cite{OR75}, and Haydon--Odell--Rosenthal \cite{HOR}. Theorem~A and Theorem~B from \cite[pages 1--2]{HOR} provide evidence for this connection, as we will discuss in Section~\ref{discussion} below. The first theorem pertains to properties within a {\em model}, while the second addresses properties of a {\em theory}. Notably, Theorem~B of \cite{HOR} is very {\em similar to} Shelah's theorem mentioned earlier, which states that a theory is unstable if and only if it has the $IP$ or $SOP$. Below, we provide a brief explanation of this theorem.

Let $X$ be a compact metric space. We will shortly define {\em proper} subclasses $C(X) \subsetneqq B_{1/4}(X) \varsubsetneqq B_{1/2}(X)$ of (real-valued) Baire-1 functions on $X$, denoted by $B_1(X)$ (cf. Definition~\ref{subclasses} below). Suppose $(f_n)$ is a sequence of real-valued continuous functions on $X$ that converges pointwise to a function $f$. Theorem~B(a) of \cite{HOR} asserts the following: If $f \in B_1(X) \setminus B_{1/2}(X)$, then $(f_n)$ has a subsequence whose spreading model\footnote{Loosely (and possibly inaccurately) speaking, a spreading model can be considered as a sequence $(\phi(a_n, y) : n < \omega)$, where $\phi$ is a formula and $(a_n)$ is a sequence in an elementary extension of the original model, in the sense of model theory.} is equivalent to the unit vector basis of $\ell_1$.

Theorem~B(b) of \cite{HOR} states: If $f \in B_{1/4}(X) \setminus C(X)$, there exists a convex block subsequence $(g_n)$\footnote{Roughly speaking, a convex block subsequence of $(f_n)$ can be thought of as a sequence of Boolean combinations of the instances of $\phi(a_n, y)$, where $a_n$ belongs to the monster model and $\phi$ is a formula.} of $(f_n)$ whose spreading model is equivalent to the summing basis for $c_0$.

We can assume that each spreading model of a sequence $(f_n)$ corresponds to a sequence of formulas of the form $(\phi(a_n, y) : n < \omega)$, where $a_n$ belongs to the monster model of a {\em theory}. Consequently, Theorem~B(a) should be compared with Theorem~\ref{NIP=Baire-1/2}(v) below, as it corresponds to the independence property ($IP$) for a {\em theory}. Similarly, Theorem~B(b) should be compared with Proposition~\ref{SOP->c0} below, as it is related to the strict order property ($SOP$) for a {\em theory}.

However, we argue that there is no perfect analogue of Shelah's theorem in continuous logic or Banach space theory. This absence arises from the existence of functions in $B_{1/2}(X) \setminus B_{1/4}(X)$, meaning that the two alternatives (a) and (b) in Theorem~B do not cover all cases.\footnote{See Section~\ref{discussion} for a detailed discussion.}

There are several considerations to keep in mind. The notions and results in the continuous framework are not literal translations of their classical logic counterparts, and the proofs are often more complex than those in the classical setting.\footnote{For example, $NIP$ in continuous logic corresponds to Baire-1/2 functions, which strictly include $DBSC$ functions. Recall from \cite{K-Baire} that $NIP$ corresponds to $DBSC$ in classical logic.} On the other hand, not all results from the classification theory of classical ($\{0,1\}$-valued) logic have been established here. In fact, classification in the continuous case is {\em finer}.\footnote{In continuous logic, based on the fact that $DBSC \subsetneqq B_{1/2}$, there exists a classification corresponding to subclasses between $DBSC$ and $B_{1/2}$. However, in the classical case, $DBSC = B_{1/2}$. (Cf. Section~\ref{discussion}.)}

\medskip

The main results of this paper are summarized as follows:

\begin{itemize}
  \item \textbf{Theorems~\ref{SOP-main} and \ref{SOP characterization}:} These theorems extend and refine Shelah's theorem, providing a comprehensive characterization of $SOP$ in terms of function spaces.
  \item \textbf{Theorem~\ref{NIP=Baire-1/2}:} This result characterizes $NIP$ (for a continuous theory) using function spaces and Banach space theory. As a corollary (Corollary~\ref{NIP=DBSC}), it offers a characterization of $NIP$ formulas in classical logic.
  \item \textbf{Theorem~\ref{NIP definability}:} This theorem establishes the Baire-1/2 definability of $NIP$ formulas in continuous logic. The corresponding corollary (Corollary~\ref{NIP definability classic}) demonstrates the $DBSC$ definability of $NIP$ formulas in classical logic.
  \item \textbf{Additional results:} Lemma~\ref{bounded variation-continuous} generalizes the concept of alternation number in model theory. Proposition~\ref{SOP->c0} demonstrates the embedding of $c_0$ in the space of formulas within $SOP$ theories, while Proposition~\ref{Shelah-like2} presents several theorems analogous to Shelah's results. Furthermore, the discussions and conclusions in Chapter~5 shed light on significant and, in our view, foundational topics.
\end{itemize}

\medskip

The broader impact and significance of this work can be outlined as follows:

\medskip

\noindent (i) This paper represents progress toward a classification theory for continuous first-order theories, drawing a parallel to the well-established classification in classical logic.

\noindent (ii) The results presented here deepen our understanding of classical model theory and, in some cases, lead to new findings. Additionally, they elucidate why certain theorems are not applicable in continuous logic while holding in the classical setting.

\noindent (iii) This work introduces new classification frameworks for Banach spaces, addressing areas that have not previously been explored.

\noindent (iv) By forging connections between model theory, Banach space theory, and function theory, this study opens up new avenues of inquiry, introduces novel tools, and offers fresh insights into these foundational fields.

\medskip

We anticipate that readers may identify further contributions and implications beyond those highlighted here.

\medskip
This paper is organized as follows: In Section~2, we present the notation and preliminaries. Section~3 focuses on the study of $SOP$ in continuous logic and provides characterizations of $SOP$, some of which are novel even for classical logic. Section~4 examines $NIP$ in continuous logic, offering a characterization of $NIP$ in terms of function spaces and the definability of coheirs. In Section~5, we discuss why Shelah's theorem does not extend to continuous logic, the existence of pathological Banach spaces (i.e., spaces that do not contain infinite-dimensional reflexive subspaces, $\ell_1$, or $c_0$), and the relationship between these Banach spaces and the failure of Shelah's theorem for continuous logic. Finally, we introduce new classes of (continuous) theories based on function spaces.

\section{Notation and Preliminaries} \label{2}
We work within the framework of continuous first-order logic \cite{BBHU, BU}, assuming that the reader is familiar with its concepts. Continuous logic is an extension of classical ($\{0,1\}$-valued) logic; therefore, our results also apply to  the classical case. Our model theory notation is standard, and a reference such as \cite{BBHU} provides sufficient background for the model-theoretic aspects of this paper.

\begin{Convention} \label{app satisfaction} Nevertheless, in some parts-- particularly in Section~3, which deals
 more with combinatorial than analytical aspects-- we make use of the more intuitive notation \cite{HI}. Therefore, in these cases, we use the quantifiers $\forall$, $\exists$, and the connectives $\wedge$, $\vee$, and the formulas take the form $\forall x (\phi(x)\geq r)$, $\exists x(\phi(x)\leq r\wedge\psi(x)\geq s)$,   etc.

 From a semantic perspective, the notation \cite{HI}  requires {\em approximate satisfaction} in order for the compactness theorem to hold. That said, we do not explicitly discuss the notion of approximate satisfaction, although it is clearly present in some arguments, including in Fact~\ref{Ramsey}.
  Despite this, it is worth recalling that, intuitively, the approximate satisfaction  $M\models_{\text{app}}\exists x (\phi(x)\leq r)$ means that for every $\epsilon>0$ there is some $a$ in $M$ such that the statement $\phi(a)\leq r+\epsilon$ holds. Now,
   $M\models \inf_x  \phi(x)\leq r$
   if and only if  $M\models_{\text{app}}\exists x (\phi(x)\leq r)$.
This means that Henson-Iovino satisfaction of $\varphi$ is equivalent to the following: for any arbitrary $\epsilon > 0$, the simultaneous satisfaction—in the standard discrete sense—of every $\epsilon$-relaxed version of the inequalities that constitute $\varphi$.


It is worth noting for the reader that our first actual use of the notation of \cite{HI} and the notion of approximate satisfaction appears at the beginning of Section 3. For more detailed points regarding the comparison of syntax and semantics between the logics \cite{HI} and \cite{BBHU}, refer to Notation 1.9 and Remark 2.11 in \cite{BU}.
\end{Convention}

 For the functional analysis component, the reader is advised to consult this paper alongside \cite{HOR} and \cite{K-Baire}. Nonetheless, we include some necessary background from functional analysis for completeness.

\subsection{Subclasses of Baire~1 Functions}
In this subsection, we define the function spaces that will be central to our study and outline some elementary relationships between them.

Let $X$ be a set and $A$ a subset of ${\mathbb{R}}^X$. The topology of \emph{pointwise convergence} on $X$ is inherited from the usual product topology on ${\mathbb{R}}^X$. A typical neighborhood of a function $f$ is given by 
\[
U_f(x_1,\ldots,x_n;\epsilon) = \{g \in {\mathbb{R}}^X : |f(x_i) - g(x_i)| < \epsilon \text{ for } i \leq n\},
\]
where $\epsilon > 0$ and $\{x_1, \ldots, x_n\}$ is a finite subset of $X$. 

Throughout this paper, subsets of ${\mathbb{R}}^X$ are endowed with the topology of pointwise convergence unless explicitly stated otherwise.

On a general (non-compact) topological space $X$, the spaces: $C(X)$ of continuous (real-valued)
functions and $C_b(X)$ of continuous bounded functions are different. The supremum norm
$\|f\|_\infty =\sup_{x\in X} |f(x)|$ is finite for $f\in C_b(X)$, but not necessarily for $f\in C(X)$—i.e., it is
generally $C_b(X)$, and not $C(X)$, that is a Banach space.

\begin{Convention} (i): For the purposes of this article, it suffices to always consider the space $X$ as a compact Hausdorff space (and even metric). Thus, $C_b(X)=C(X)$.
Nevertheless, some of the analytical theorems we use hold under weaker assumptions than compactness for $X$. Wherever necessary, we will note this in the footnotes.
\newline
(ii): In this paper, we essentially work with the topology of pointwise convergence on $C(X)$. 
In some references, the notation $C_p(X)$ is used to emphasize this topology; however, 
we continue to use the notation $C(X)$ in order to maintain consistency with one of our 
main references, namely \cite{BFT}.
On the other hand, in the reference used in this paper, pointwise boundedness is sufficient. 
However, for the sake of simplicity and since formulas in continuous logic take values in 
$[0,1]$, we adopt the stronger assumption of uniform boundedness.
\newline
(iii)
In the following definition, we have used Baire’s original definition of Baire class~1 functions, 
although the standard modern definition is different (see Definition~1A(c) in \cite{BFT}). 
However, since we assume that $X$ is compact and metrizable, Proposition~1E in \cite{BFT}
implies that these two definitions coincide.
\end{Convention}

\begin{Definition} \label{subclasses} 
Let \( X \) be a compact metric space. The following is a list of spaces of real-valued functions defined on \( X \). The key spaces to note are (iii) and (iv).

\begin{itemize}
    \item[(i)] \( C(X) \): The space of continuous real-valued functions on \( X \).
    \item[(ii)] \( B_1(X) \): The first Baire class, consisting of functions \( f: X \to \mathbb{R} \) such that \( f \) is the pointwise limit of a sequence of continuous functions.
    \item[(iii)] \( DBSC(X) \)\footnote{Difference of Bounded Semi-Continuous functions}: The space of real-valued functions \( f \) on \( X \) such that there exist bounded semi-continuous functions \( F_1, F_2 \) with \( f = F_1 - F_2 \).
    \item[(iv)] \( B_{1/2}(X) \) (Baire-1/2): The space of real-valued functions \( f \) on \( X \) such that \( f \) is the uniform limit of a sequence \( (F_n) \in DBSC(X) \).
    \item[(v)] \( B_{1/4}(X) \) (Baire-1/4): The space of real-valued functions \( f \) on \( X \) for which there exist \( C > 0 \) and a sequence \( (F_n) \in DBSC(X) \) such that:
    \begin{enumerate}
        \item \( (F_n) \) uniformly converges to \( f \),
        \item for all \( n \), there is a sequence \( (f_i^n) \in C(X) \) with \( f_i^n \to F_n \) pointwise and 
        \[
        \sum_{i=1}^\infty |f_i^n - f_{i+1}^n(x)| \leq C \quad \text{for all } x \in X.
        \]
    \end{enumerate}
\end{itemize}
\end{Definition}

\medskip

We now provide some elementary relationships between these function spaces. (However, we will rely on additional content from \cite{HOR}.)

\begin{Fact} \label{properties} 
\begin{itemize}
    \item[(i)] \cite[page~3]{HOR} A function \( f \in DBSC(X) \) if and only if there exist a uniformly bounded sequence \( (f_n) \in C(X) \) and \( C > 0 \) such that \( f_n \to f \) pointwise and 
    \[
    \sum_{n=1}^\infty |f_n(x) - f_{n+1}(x)| \leq C \quad \text{for all } x \in X.
    \]

    \item[(ii)] \cite[page~3]{HOR} The spaces \( DBSC(X) \), \( B_{1/4}(X) \), and \( B_{1/2}(X) \) are Banach algebras with respect to suitable norms.

    \item[(iii)] \cite[Proposition~5.1]{HOR} For an uncountable compact metric space \( X \),
    \[
    C(X) \subsetneqq DBSC(X) \subsetneqq B_{1/4}(X) \subsetneqq B_{1/2}(X) \subsetneqq B_1(X).
    \]

    \item[(iv)] \cite[Proposition~2.2]{CMR} Let \( f \) be a simple\footnote{A function is called \emph{simple} if its range is a finite set.} real-valued function on \( X \). Then \( f \in B_{1/2}(X) \) if and only if \( f \in DBSC(X) \).
\end{itemize}
\end{Fact}

\medskip

Let \( Y \) be a topological space. A subset \( A \subseteq Y \) is said to be \emph{relatively compact} if it has compact closure in \( Y \). We recall from \cite[Fact~3.1]{K-Baire} the well-known compactness theorem of Eberlein and \v{S}mulian.

\begin{Fact}[Eberlein--\v{S}mulian Theorem] \label{ES theorem}
Let \( X \) be a compact Hausdorff space and \( A \) a norm-bounded subset of \( C(X) \). Then for the topology of pointwise convergence, the following are equivalent:
\begin{itemize}
    \item[(1)] \( A \) is relatively compact in \( C(X) \).
    \item[(2)] The following two properties hold:
    \begin{itemize}
        \item[(RSC)]\footnote{Relative sequential compactness in \( \mathbb{R}^X \)} Every sequence in \( A \) has a convergent subsequence in \( \mathbb{R}^X \).
        \item[(SCP)]\footnote{Sequential completeness} The limit of every convergent sequence in \( A \) is continuous.
    \end{itemize}
\end{itemize}
\end{Fact}

\noindent \emph{Explanation.} See \cite{W} for a short proof of the Eberlein--\v{S}mulian Theorem. Recall from \cite[Lemma 3.56]{FH} that for each compact Hausdorff space \( X \), the weak topology and the pointwise convergence topology on \( C(X) \) are the same.\footnote{In fact, a similar statement holds for $C_b(X)$, under the weakest assumption that $X$ is merely countably compact. However, as we mentioned earlier, for the purposes of this article, it suffices to work with a compact and even metric space $X$.}
 Notice that every sequence contains a subsequence converging to an element of \( C(X) \) if and only if:
\begin{itemize}
    \item[(i)] Every sequence has a convergent subsequence in \( \mathbb{R}^X \), and
    \item[(ii)] The limit of every convergent sequence is continuous.
\end{itemize}
Thus, condition (2) is precisely the condition \( B \) in the main theorem of \cite{W}, while condition (1) corresponds to condition \( A \) in \cite{W}. Although the Eberlein--\v{S}mulian Theorem is proved for arbitrary Banach spaces (even normed spaces), it follows easily from the case of \( C(X) \) (see \cite[Theorem~462D]{Fremlin4}--Version of 1.2.13).\footnote{See https://www1.essex.ac.uk/maths/people/fremlin/cont46.htm}

It is worth mentioning that, what is actually used in Fact~\ref{ES theorem}
is taken from Theorems 1 and 2 in Section 3 of  Grothendieck's paper \cite{Gro}.   Grothendieck, in that paper, refers to \v{S}mulian's work on Banach spaces, but his own results are specifically about \( C_p \) spaces (i.e., spaces of continuous functions endowed with the topology of pointwise convergence).

\medskip

We now recall a crucial result of Bourgain, Fremlin, and Talagrand \cite{BFT}. First, recall that a subset \( A \) of a topological space \( Y \) is called \emph{relatively countably compact} if every sequence in \( A \) has a cluster point in \( Y \).

\begin{Definition} \label{angelic}
A regular Hausdorff space \( Y \) is said to be \emph{angelic} if:
\begin{itemize}
    \item[(i)] Every relatively countably compact set \( A \) is relatively compact, and
    \item[(ii)] Every point in the closure of a relatively compact set \( A \) is the limit of a sequence in \( A \).
\end{itemize}
\end{Definition}

\begin{Fact}[Bourgain--Fremlin--Talagrand Criterion] \label{BFT criterion}
Let \( X \) be a compact Hausdorff space and \( A \) a \emph{countable}, norm-bounded subset of \( C(X) \). Then for the topology of pointwise convergence, the following are equivalent:
\begin{itemize}
    \item[(1)] \( A \) is relatively compact in \( B_1(X) \).
    \item[(2)] The closure of \( A \) in \( \mathbb{R}^X \) is angelic.
    \item[(3)] If \( r < s \) and \( (f_n) \) is a sequence in \( A \), then there exist disjoint \emph{finite} sets \( E, F \) such that:
    \[
    \bigcap_{n \in E}[f_n \leq r] \cap \bigcap_{n \in F}[f_n \geq s] = \emptyset.
    \]
    (Here \( [f \leq r] = \{x \in X : f(x) \leq r\} \) and \( [f \geq s] = \{x \in X : f(x) \geq s\} \).)
\end{itemize}
\end{Fact}

\begin{proof}
The equivalence (2)~\(\Leftrightarrow\)~(3) follows from the equivalence (vi)~\(\Leftrightarrow\)~(viii) of \cite[Theorem~4D]{BFT}. Indeed, since \( X \) is compact, we can work with \emph{infinite} subsets \( E, F \subseteq \mathbb{N} \), so the condition (3) above and (viii) of Theorem~4D in \cite{BFT} are equivalent (see also the equivalence (iii)~\(\Leftrightarrow\)~(iv) of \cite[Lemma~3.12]{K3}).

(2)~\(\Rightarrow\)~(1): This is evident, as the limits of sequences of continuous functions are Baire-1.

(1)~\(\Rightarrow\)~(2): By Theorem~3F of \cite{BFT}, \( B_1(X) \) is angelic. Since any subspace of an angelic space is angelic (cf.  \cite[462B(a)]{Fremlin4}--Version of 1.2.13)\footnote{See the following for this version: https://www1.essex.ac.uk/maths/people/fremlin/cont46.htm}, the closure of \( A \) in \( \mathbb{R}^X \) is angelic. Note that, since the closure of \( A \) in \( B_1(X) \) is compact, this closure is closed (and compact) in the space \( \mathbb{R}^X \). In particular, this implies that the closure of \( A \) in \( \mathbb{R}^X \) is contained in \( B_1(X) \).
\end{proof}
In Fact \ref{BFT criterion}, the countability condition $A$ is not necessary (refer to Corollary 4G of \cite{BFT}). However, as we will  see in the subsection below, the set $A$ is considered as a parameter subset of the space $X$. Since, to utilize the results of \cite{HOR}, we assume $X$ is compact and metric, it is preferable to impose the countability condition on $A$.

\subsection{Continuous Model Theory}
As previously mentioned, the model-theoretic notations, whenever we are discussing analysis rather than combinatorics, are the same as those presented in \cite{BBHU}.

\begin{Convention}
In this paper, when the formula $\phi(x,y)$ is discussed, it is understood that $x$ and $y$ are $m$-tuples and $n$-tuples for some finite $m$ and $n$. Although in some arguments it suffices to assume $|x|= 1$, we do not impose this restriction. Thus, if $x=(x_1,\ldots,x_m)$, by $\sup_x\phi(x)$ (or $\forall x\phi(x)$) we mean that
 $\sup_{x_1}\cdots\sup_{x_m}\phi(x_1,\ldots,x_m)$
 (or $\forall{x_1}\cdots\forall{x_m}\phi(x_1,\ldots,x_m)$), etc. Nevertheless, whenever emphasis is needed to avoid ambiguity, we explicitly write $\bar x$ or $(x_1,\ldots,x_m)$.
\end{Convention}

 We fix an \( L \)-formula \( \phi(x,y) \), a complete \( L \)-theory \( T \), and a subset \( A \) of the monster model of \( T \). Define \( \tilde{\phi}(y,x) = \phi(x,y) \). 

We define \( p = \mathrm{tp}_\phi(a/A) \) as the function \( \phi(p,y):A \to [0,1] \) given by \( b \mapsto \phi(a,b) \). This function is called a complete \(\phi\)-type over \( A \). The set of all complete \(\phi\)-types over \( A \) is denoted by \( S_\phi(A) \). 

We equip \( S_\phi(A) \) with the weakest topology such that all functions \( p \mapsto \phi(p,b) \) (for \( b \in A \)) are continuous. The resulting space is compact and Hausdorff, but not necessarily totally disconnected. Let \( X = S_{\tilde{\phi}}(A) \), the space of complete \(\tilde{\phi}\)-types over \( A \). 

Note that the functions \( q \mapsto \phi(a,q) \) (for \( a \in A \)) are \emph{continuous}. As \( \phi \) is fixed, we can identify this set of functions with \( A \). Thus, \( A \) can be seen as a subset of all bounded continuous functions on \( X \), denoted \( A \subseteq C(X) \). Similarly, we can define the spaces \( B_1(X) \), \( DBSC(X) \), \( B_{1/2}(X) \), and \( B_{1/4}(X) \) as previously described.

The following result is worth highlighting (see \cite[Corollary~2.10]{K3} and \cite[Proposition~2.6]{Iovino} for real-valued logics, and \cite[Proposition~2.2]{Pillay-Grothendieck} for classical logic):

\begin{Fact}[Grothendieck Criterion] \label{EG}
Let \( (a_i) \) be a sequence in some model of \( T \), and let \( \phi(x,y) \) be a formula. Then the following are equivalent:
\begin{itemize}
    \item[(i)] There does not exist a sequence \( (b_j) \) and constants \( r < s \) such that either:
    \begin{itemize}
        \item \( \phi(a_i,b_j) \leq r \wedge \phi(a_j,b_i) \geq s \) holds if and only if \( i < j \), or
        \item \( \phi(a_j,b_i) \leq r \wedge \phi(a_i,b_j) \geq s \) holds if and only if \( i < j \).
    \end{itemize}
    \item[(ii)] For any sequence \( (b_j) \),
    \[
    \lim_i \lim_j \phi(a_i,b_j) = \lim_j \lim_i \phi(a_i,b_j),
    \]
    whenever both limits exist.
    \item[(iii)] Every function in the pointwise closure of 
    \[
    A = \{\phi(a_i,y) : S_{\tilde{\phi}}(\{a_i\}) \to [0,1] : i < \omega\}
    \]
    is continuous. Equivalently, \( A \) is relatively pointwise compact in \( C(S_{\tilde{\phi}}(\{a_i\})) \).
\end{itemize}
\end{Fact}

\begin{Remark}
It is worth mentioning that Fact~\ref{EG} is derived from one of the results in Grothendieck's seminal paper \cite{Gro}. 
In particular, the equivalence of (ii) and (iii) has been proven in a much more general setting for 
the space of continuous functions with the topology of pointwise convergence  \( C_p(X) \), where \( X \) can be  {\bf countably}   compact and Hausdorff. 
Grothendieck mentions there that he used ideas from Eberlein in this achievement; therefore, it may also be permissible to refer to this fact as the Eberlein–Grothendieck  (fact/theorem).
\end{Remark}

\section{Strict Order Property}

In this section and the next, we introduce and study the ``correct" generalizations of \(SOP/IP\) for continuous logic. We argue that they are the appropriate generalizations for our purposes.

\medskip
First, we need some notions and facts about indiscernible sequences. Let \( \phi(x,y) \) be a formula, \( n \) a natural number, and \( A \subseteq [0,1] \). We say that a condition \( \psi(x_1, \ldots, x_n) \) is a \(\phi\)-\(n\)-\(A\)-condition if it is of the form
\[
\exists y \left( \bigwedge_{i\in E} \phi(x_i, y) \leq r_i \wedge \bigwedge_{j\in F} \phi(x_j, y) \geq r_j \right)
\]
or
\[
\forall y \left( \bigvee_{i\in E} \phi(x_i, y) \leq r_i \vee \bigvee_{j\in F} \phi(x_j, y) \geq r_j \right),
\]
where  $E$ and $F$ are disjoint subsets of $\{1,\ldots,n\}$ and 
\( r_i, r_j \in A \). In this case, \( \psi(x_1, \ldots, x_n) \) has \( n \) free variables \( x_1, \ldots, x_n \) and a bounded variable \( y \).

If \( M \) is a model of a theory and \( \bar{a} = (a_1, \ldots, a_n) \in M^n \), the \(\phi\)-\(n\)-\(A\)-type of \( \bar{a} \), denoted \( tp_{\phi,n,A}(\bar{a}) \), is the set of all \(\phi\)-\(n\)-\(A\)-conditions \( \psi(\bar{x}) \) such that \( \models \psi(\bar{a}) \).

\begin{Remark}
As stated in Convention~\ref{app satisfaction}, we use the notation \cite{HI} in combinatorial discussions. We note that it would have been possible to present the notion of $\phi$-$n$-$A$-type in the same way types were defined in the previous section; however, for several reasons, it is preferable not to do so. The first and more important reason is that in the proof of the main theorem of this section (namely, Theorem~\ref{SOP-main}), this notation is significantly more useful. Secondly, we only need part of the information about the  tuple $\bar a$ in the type $tp_{\phi,n,A}(\bar a)$,  not all of it—that is, $tp_{\phi,n,A}(\bar a)$ is a partial type, not a complete one.
\end{Remark}

\begin{Definition}
{\em Let \( T \) be a complete \( L \)-theory, \( \phi(x,y) \) an \( L \)-formula, \( n \) a natural number, \( A \subseteq [0,1] \), and \( (a_i) \) a sequence in some model. We define the following:
\begin{itemize}
    \item[(i)] We say that the sequence \( (a_i) \) is a \emph{\(\phi\)-\(n\)-\(A\)-indiscernible sequence} (over the empty set) if for each \( i_1 < \cdots < i_n < \omega \), \( j_1 < \cdots < j_n < \omega \), 
    \[
    tp_{\phi,n,A}(a_{i_1}, \ldots, a_{i_n}) = tp_{\phi,n,A}(a_{j_1}, \ldots, a_{j_n}).
    \]
    \item[(ii)] We say that \( (a_i) \) is a \emph{\(\phi\)-\(n\)-indiscernible sequence} if it is \(\phi\)-\(n\)-\(A\)-indiscernible for \( A = [0,1] \).
    \item[(iii)] We say that \( (a_i) \) is a \emph{\(\phi\)-indiscernible sequence} if it is \(\phi\)-\(n\)-indiscernible for all \( n \in \mathbb{N} \).
\end{itemize}
}
\end{Definition}

In classical logic, the following result is well-known (cf. Theorem~I.2.4 of \cite{Shelah}), although there are technical considerations in the continuous case.

\begin{Fact} \label{Ramsey}  
Let \( T \) be a complete \( L \)-theory, \( \phi(x,y) \) an \( L \)-formula, \( M \) a model of \( T \), and \( n \) a natural number.

\begin{itemize}
    \item[(i)] If \( (a_i) \) is an \emph{infinite} sequence in \( M \), there is an infinite \(\phi\)-indiscernible sequence \( (c_i) \) (possibly in an elementary extension of \( M \)) such that for every finite set \( A \subseteq [0,1] \cap \mathbb{Q} \) and every \( k \in \mathbb{N} \), there is a \(\phi\)-\(k\)-\(A\)-indiscernible subsequence \( (b_i) \subseteq (a_i) \) such that
    \[
    tp_{\phi,k,A}(c_1, \ldots, c_k) = tp_{\phi,k,A}(b_1, \ldots, b_k).
    \]
    \item[(ii)] If \( I \subset J \) are two (infinite) linear ordered sets and \( (a_i)_{i \in I} \) is an \emph{infinite} \(\phi\)-indiscernible sequence in \( M \), there is a sequence \( (b_j)_{j \in J} \) (possibly in an elementary extension of \( M \)) which is a \(\phi\)-indiscernible sequence and $(b_j)_{j\in J}$ has the same $\phi$-$k$-$A$-type as $(a_i)_{i\in I}$ (for every finite $A\subseteq[0,1]\cap\Bbb Q$ and every $k\in \Bbb N$).
\end{itemize}
\end{Fact}

\begin{proof}  
(i): As Ramsey's theorem (Theorem~2.1, appendix 2 of \cite{Shelah}) works for finitely many types (colors), and continuous logic is \( [0,1] \)-valued, we need to initially work with \( A \)-valued logics, where \( A \) is a finite subset of \( [0,1] \), and then generalize it to the \( [0,1] \)-valued case. 

Let \( (A_k) \) be an increasing sequence of finite subsets of rational numbers in \( [0,1] \) such that their union is \( [0,1] \cap \mathbb{Q} \). Specifically, \( A_k \) is finite, \( A_k \subseteq A_{k+1} \), and \( \bigcup_k A_k = [0,1] \cap \mathbb{Q} \).

By induction on \( k \), we obtain a sequence \( ((a_i^k): k < \omega) \) of subsequences of \( (a_i) \) as follows. Let \( (a_i^0) = (a_i) \), and suppose that \( (a_i^{k-1}) \) is given. By Ramsey's theorem, there exists a subsequence \( (a_i^k) \) which is \(\phi\)-\(k\)-\(A_k\)-indiscernible. (Note that Ramsey's theorem works because \( A_k \) and \( k \) are finite, so \(\phi\)-\(k\)-\(A_k\)-types are finite.)

Now, by the compactness theorem, there exists a sequence \( (c_i) \) (possibly in a saturated model) such that for all \( k \in \mathbb{N} \), \( (c_i) \) is \(\phi\)-\(k\)-\(A_k\)-indiscernible, and
\[
tp_{\phi,k,A_k}(c_1, \ldots, c_k) = tp_{\phi,k,A_k}(a_1^k, \ldots, a_k^k).
\]
Since \( \bigcup_k A_k \) is dense in \( [0,1] \cap \mathbb{Q} \), the proof is complete.

\noindent (ii) follows from the compactness theorem.
\end{proof}

The above result can be presented for a more general state, but that is enough for our purpose.

\medskip
In the following, we introduce the notion \(SOP\) for a continuous theory. We remind that the interpretation of \(\phi(x,y)\) is a function into the interval \([0,1]\), and identifying \emph{True} with the value zero and \emph{False} with one, the following notion is a generalization of the \(SOP\) in classical (\(\{0,1\}\)-valued) model theory.

\begin{Definition}[$SOP$ for a continuous theory]  \label{SOP-definition}
{\em  Let \(T\) be a continuous theory and \(\mathcal{U}\) the monster
model of \(T\).

\noindent (i) Let \(\phi(x,y)\) be a formula and \(\epsilon>0\).
 We say the formula \(\phi(x,y)\) (in continuous logic) has the {\em
  \(\epsilon\)-strict order property (\(\epsilon\)-SOP)} if there exists a sequence
\((a_i b_i : i < \omega)\) in  \(\mathcal{U}\) such that for all \(i<j\),
$$\phi(\mathcal{U},a_i)\leqslant\phi(\mathcal{U},a_{i+1})\ \ \ \mbox{ and } \ \ \phi(b_j,a_i)+\epsilon\leqslant\phi(b_{i+1},a_j).$$
(Here, by $\phi(\mathcal{U},a_i)\leqslant\phi(\mathcal{U},a_{i+1})$, we mean $\phi(b,a_i)\leqslant\phi(b,a_{i+1})$ for all $b\in\mathcal U$.)
\medskip\noindent
(ii) We say the formula \(\phi(x,y)\) (in continuous logic) has the
{\em strict order property} (\(SOP\)) if it has \(\epsilon\)-$SOP$ for
some \(\epsilon>0\).

\medskip\noindent
(iii) We say that \(T\) has \(SOP\) if there is a formula \(\phi(x,y)\)
which has \(SOP\). We say that $T$ has $NSOP$ if it does not have   $SOP$.}
\end{Definition}

Although, unlike stability and $NIP$, there is no universal consensus on the definition of $NSOP$ in continuous logic, and several options for this concept exist, Remark~\ref{main remark} and the results of this section convince us that Definition~\ref{SOP-definition} is a suitable option for this concept in continuous logic.

\begin{Remark} \label{main remark}
	(i) It is easy to verify that, in classical logic, a theory has \(SOP\) (in the usual sense) if and only if it has  \(SOP\) (as in Definition~\ref{SOP-definition}).  Indeed, suppose that Definition~\ref{SOP-definition}(i) holds for \(\phi(x,y)\), \((a_i b_i : i < \omega)\), and \(\epsilon > 0\). Then, \(\phi(\mathcal{U}, a_i) \supseteq \phi(\mathcal{U}, a_{i+1})\) for all \(i\). For \(j = i+1\), as the formula is \(\{0,1\}\)-valued, we have \(\epsilon \leq \phi(b_{i+1}, a_{i+1})\), and so \(\models \neg\phi(b_{i+1}, a_{i+1})\). Therefore, \(\phi(b_{i+1}, a_i) \leq 1 - \epsilon\), and so \(\models \phi(b_{i+1}, a_i)\). To summarize, \(b_{i+1} \in \phi(\mathcal{U}, a_i) \setminus \phi(\mathcal{U}, a_{i+1})\) and  \(\phi(\mathcal{U}, a_i) \supsetneq \phi(\mathcal{U}, a_{i+1})\) for all \(i\). The converse is similar.

	\smallskip
	(ii) In \cite[Question~4.14]{Ben2}, Ben Yaacov introduced a notion of \(SOP\) for continuous logic and claimed that Shelah's theorem holds using this notion. As mentioned before, we will argue that this claim does not hold. 
	On the other hand, almost two years after a version of this article was submitted to publish, we learned that James Hanson studied the notion of \(SOP\) in the sense of \cite{Ben2} in his thesis \cite[pages 485-490]{Hanson thesis} using the notion of quasi-metrics.\footnote{It is well-known that quasi-metrics are   the continuous version of pre-orders.}
	In fact, he  showed that Ben Yaacov's notion is equivalent to instability.\footnote{Hanson informed us of this equivalence. Notice that he uses continuous connectives to define a quasi-metric with an infinite chain, even in the classical case. Notice that Definition~3.3 above is also strong because it is defined for `formulas' but not `definable' predicates.}
	
	\smallskip
	(iii) It is easy to show that our notion (Definition~\ref{SOP-definition}) is strictly stronger than instability. Indeed, suppose that the formula \(\phi\) has \(\epsilon\)-SOP as above. That is, there exists a sequence
	\((a_i b_i : i < \omega)\) in  \(\mathcal{U}\) such that for all \(i < j\),
	\(\phi(\mathcal{U}, a_i) \leqslant \phi(\mathcal{U}, a_{i+1})\)  and \(\phi(b_j, a_i) + \epsilon \leqslant \phi(b_{i+1}, a_j)\).
	Set \(\rho(x,y) := \sup_z (\phi(x,y) \dotminus \phi(y,z))\). Then \(\rho(a_i, a_j) = 0\) and \(\rho(a_j, a_i) \geq \epsilon\) for all \(i < j\). Therefore, \(\rho\) is a quasi-metric with an infinite \(\epsilon\)-chain (in the sense of \cite{Hanson thesis}). Define \(\rho'(x,y) := \rho(x,y) \dotplus \cdots \dotplus \rho(x,y)\) (\(n\)-times), for some \(n \geq 1/\epsilon\). Then \(\rho'(a_i, a_j) = 0\) and \(\rho'(a_j, a_i) = 1\) for all \(i < j\), i.e. \(\rho'\) has an infinite \(0\)-chain.\footnote{It can be easily shown that the converse is also true. That is, if a `formula' \(\psi(x,y)\) defines a quasi-metric with infinite \(0\)-chain, then the theory has \(SOP\) (as in Definition~3.3 above).} This shows, by Proposition B.5.5 in \cite{Hanson thesis},  that \(T\) has \(SOP_n\) for all \(n \geq 3\).\footnote{Recall that Hanson showed that a theory is unstable if and only if there exists a definable quasi-metric with an infinite \(1/2\)-chain, and is \(SOP_n\) (\(n \geq 3\)) if there exists a definable quasi-metric with an infinite \(r\)-chain with \(r < \frac{1}{n-1}\). Therefore, the parameter \(r\) is actually important.}  In \cite{CT}, it is shown that the theory of the Urysohn space has \(SOP_n\) (for \(n \geq 3\)).\footnote{They also claim that this theory has \(NSOP_\infty\), but it seems that their proof is not correct.}
	Also, it is not hard to show that this theory has \(NSOP\). (In fact, it has \(SCP\) as in \cite{K5}.) This means  that this theory is \(NSOP\) in our sense, and \(SOP_n\) (for \(n \geq 3\)) in the sense of  \cite{CT}.  

	\smallskip
	(iv) Alternatively, if \(\phi(x,y)\) satisfies the conditions in Definition~\ref{SOP-definition} for some \(\epsilon>0\), then using compactness the following formula implies \(SOP_n\) (for all \(n \geq 3\)): 
	$$\max\left(\sup_x\left(\phi(x,x_1)\dotminus \phi(x,x_2)\right), \ \epsilon\dotminus \sup_x\left(\phi(x,x_2)\dotminus \phi(x,x_1)\right)\right).$$

	\smallskip
	(v) In \cite{K5}, we introduced a weak version of \(SOP\) for continuous logic (denoted by \(wSOP\)) and showed that a variant of Shelah's theorem holds using this notion. This result is not a perfect analogy of Shelah's theorem as the notion \(wSOP\) is stated by a sequence of formulas, but not only one formula.
	(Notice that the notion \(wSOP\) is strictly weaker than \(SOP\) as in Definition~\ref{SOP-definition}.) It is easy to verify  that any alternative of \(SOP\) for continuous logic should be stronger than the notion \(SOP\) introduced in this article, as we use continuous quantifiers and the finer structure of continuous models.

\smallskip
(vi) In a recent paper \cite{KP-random}, it was shown that the Shelah theorem holds for randomization theories, as continuous theories, if Definition~\ref{SOP-definition} is considered for $SOP$ in continuous logic. This is another confirmation of the suitability of this notion.
\end{Remark}

\begin{Definition} \label{uniformly blocked}  
{\em (i) We\footnote{The terminology was borrowed from \cite{Baldwin-report}.} say that {\em the independence property is uniformly blocked for $\phi(x,y)$ on $(a_i)$}, if there is a natural number $N$ such that for each real numbers $r<s$ there is a set $E \subseteq \{1, \ldots, N\}$ (depending only on  $r,s$) such that for each $i_1 < \cdots < i_N < \omega$, the following does not hold:
$$ \exists x \big( \bigwedge_{j \in E} \phi(x, a_{i_j}) \leq r \wedge \bigwedge_{j \in N \setminus E} \phi(x, a_{i_j}) \geq s \big). $$

\medskip
\noindent (ii) We say that {\em $(b_i), (a_i)$ witness the order property with $r < s$} if $\phi(b_j, a_i) \leq r$ for all $i<j<\omega$ and $\phi(b_j, a_i) \geq s$ for all $j\leq i < \omega$.}
\end{Definition}

\medskip
The following result generalizes and refines a crucial theorem of Shelah that every unstable $NIP$ theory has $SOP$.

\begin{Theorem} \label{SOP-main} 
Let $T$ be a complete continuous theory. The following are equivalent:

\medskip\noindent
(i) $T$ has $SOP$.

\medskip\noindent
(ii) There are a formula $\phi(x,y)$, sequences $(a_i)$, $(b_i)$, and real numbers $r < s$ such that

\begin{itemize}
    \item [(1)] The independence property is uniformly blocked for $\phi(x,y)$ on $(a_i)$, and

    \item [(2)] $(b_i), (a_i)$ witness the order property with $r < s$.
\end{itemize}
\end{Theorem}

\noindent Before giving the proof, let us remark:

\begin{Remark} \label{main-remark}
The condition (ii)(1) is related to a subclass of Baire-1 functions. (See Lemma~\ref{bounded variation-continuous} below.) This property is local and, in one sense, weaker than $NIP$ (it can hold for $IP$ formulas) and, in another sense, stronger than $NIP$ (not every $NIP$ formula satisfies it). In contrast with continuous logic, $NIP$ implies (ii)(1) in the classical case. We will see shortly that the latter is the reason why one should not expect a result similar to Shelah's theorem for continuous logic.

As Theorem~\ref{SOP-main} is a generalization of Theorem~2.6 of~\cite{K-Baire}, for easier reading, see the proof of the latter. The proofs are basically the same, although the continuous case is more complex and has one more key point.
\end{Remark}

\begin{proof}[Proof of Theorem \ref{SOP-main}]
(i)~$\Rightarrow$~(ii):
Suppose that $\phi(x,y)$ has $\epsilon$-$SOP$ with $(b_i), (a_i)$ as witnesses.
Initially, using Ramsey's theorem and compactness, we can assume that the numbers $r<s$ (with $\epsilon=s-r$) exist such that 
$\phi(b_j, a_i) \leq r$ for all $i<j$ and $\phi(b_j, a_i) \geq s$ for all $j\leq i$.
(In fact, a sequence $(b_i',a_i':i<\omega)$ can be found such that  simultaneously witness $SOP$ and there exist such $r<s$ satisfying the $OP$, but we continue to use the notation $(b_i,a_i:i<\omega)$. For the proof of this, refer to Lemma~7.2 in \cite{BU}.) Therefore, $(b_i)$, $(a_i)$ witness the order property for $\phi(x,y)$ with $r<s$.
 Since the sequence $(\phi(x,a_i))$ is increasing, let $N = \{1, 2\}$ and $E = \{2\}$. It is easy to see that $E$ is a witness to uniform $IP$ blocking.
 Therefore, the conditions (1) and (2) hold.

\medskip\noindent
(ii)~$\Rightarrow$~(i)
Similar to the argument in the proof of  Fact~\ref{Ramsey}, it can be shown that there exists a sequence $(b_i',a_i':i<\omega)$ corresponding to tuples of $x$ and $y$ such that both instances of condition (ii) are satisfied, and, moreover, $(a_i' : i < \omega)$ is $\tilde\phi$-indiscernible. (For this purpose, one must employ Ramsey's theorem and compactness for infinitely many tuples of $x$ and $y$ simultaneously.)
For the sake of simplicity and to preserve the notation, we continue to assume that the sequence $(b_i,a_i:i<\omega)$ satisfies these properties.

As the independence property is uniformly blocked for $\phi(x,y)$ on $(a_i)$, there is a natural number $N$ such that the conditions of Definition~\ref{uniformly blocked} hold for $(a_i)$.
Therefore, for any $s'$ in the interval $(r,s)$, there are some integer $n$ and $\eta : N \rightarrow \{0,1\}$ such that $\bigwedge_{i < N} \phi(x, a_i)^{\eta(i)}$ is inconsistent, where for a formula $\varphi$, we use the notation $\varphi^1$ to mean $\varphi \leq r$ and $\varphi^0$ to mean $\varphi \geq s'$. (Recall that unlike classical model theory, in continuous logic True is 0 and False is 1.) Starting with that formula, we change one by one instances of $\phi(x, a_i) \geq s' \wedge \phi(x, a_{i+1}) \leq r$ to $\phi(x, a_i) \leq r \wedge \phi(x, a_{i+1}) \geq s'$. Finally, we arrive at a formula of the form $\bigwedge_{i < n} \phi(x, a_i) \leq r \wedge \bigwedge_{n \leq i < N} \phi(x, a_i) \geq s'$. The tuple $b_n$ satisfies that formula.

Therefore, for such $r < s'$, there is some $i_0 < N$, $\eta_0 : N \rightarrow \{0,1\}$ such that
\[
(*) \ \ \ \ \bigwedge_{i \neq i_0, i_0 + 1} \phi(x, a_i)^{\eta_0(i)} \wedge \phi(x, a_{i_0}) \geq s' \wedge \phi(x, a_{i_0 + 1}) \leq r
\]
is inconsistent, but
\[
(**) \ \ \ \ \bigwedge_{i \neq i_0, i_0 + 1} \phi(x, a_i)^{\eta_0(i)} \wedge \phi(x, a_{i_0}) \leq r \wedge \phi(x, a_{i_0 + 1}) \geq s'
\]
is consistent. Let us define $\varphi_{s'}(x, \bar{a}) = \bigwedge_{i \neq i_0, i_0 + 1} \phi(x, a_i)^{\eta_0(i)}$.

The new point is that by (ii)(1), since the number of Boolean combinations of the length $N$ of formulas is finite, there are {\em infinitely} many $r + \frac{1}{n} = s' < s$ such that there exists a fixed pattern similar to $(*)$, $(**)$ where $(*)$ is inconsistent and $(**)$ is consistent for all such $s'$'s. For simplicity, we can assume from now on that $(*)$, $(**)$ is this pattern. Also, we fix an $s_0$ with this pattern and consider $\varphi_{s_0}(x, \bar{a})$ defined as above. 

Note that for these $s'$'s if $s' \searrow r$, then the statement
\[
(I) \ \ \ \ \varphi_{s_0}(x, \bar{a}) \wedge \phi(x, a_{i_0}) \geq s' \wedge \phi(x, a_{i_0 + 1}) \leq r
\]
is still inconsistent. (For inconsistency of $(I)$, note that for $s' < s_0$, since the pattern is fixed, $\varphi_{s'}(x, \bar{a}) \wedge \phi(x, a_{i_0}) \geq s' \wedge \phi(x, a_{i_0 + 1}) \leq r$ is inconsistent, and therefore it is easy to check that $\varphi_{s_0}(x, \bar{a}) \wedge \phi(x, a_{i_0}) \geq s' \wedge \phi(x, a_{i_0 + 1}) \leq r$ is inconsistent. Of course, in continuous logic, we cannot conclude that $\varphi_{s_0}(x, \bar{a}) \wedge \phi(x, a_{i_0}) \geq r \wedge \phi(x, a_{i_0 + 1}) \leq r$ is inconsistent.)

By Fact~\ref{Ramsey}(ii), increase the sequence $(a_i : i < \omega)$ to an $\tilde\phi$-indiscernible sequence $(a_i : i \in \mathbb{Q})$. Then for $i_0 \leq i < i' \leq i_0 + 1$, the formula $\varphi_{s_0}(x, \bar{a}) \wedge \phi(x, a_i) \leq r \wedge \phi(x, a_{i'}) \geq s_0$ is consistent, but $\varphi_{s_0}(x, \bar{a}) \wedge \phi(x, a_i) \geq s' \wedge \phi(x, a_{i'}) \leq r$ is inconsistent, for all $s' > r$ where $s_0 > s'$. Thus, the formula $\psi(x,y) = \varphi_{s_0}(x, \bar{a}) \wedge \phi(x,y)$ has $\epsilon$-$SOP$ (with $\epsilon=s_0-r$).
\end{proof}

We aim to provide a characterization of $SOP$ in terms of function spaces. First, we introduce a lemma that generalizes and abstracts the alternation number in model theory for continuous logic. This lemma is analogous to Lemma~2.8 in \cite{K-Baire}.

\begin{Lemma} \label{bounded variation-continuous}  
Let $(f_i)$ be a sequence of $[0,1]$-valued functions on a set $X$. Suppose that the independence property is uniformly blocked for $(f_i)$. That is, there exists a natural number $N$
 such that for each real numbers $r<s$ 
 there is 
  a set $E \subseteq \{1, \ldots, N\}$  (depending only on $r,s$) such that for each $i_1 < \cdots < i_N < \omega$,
 we have 
\[
\bigcap_{j \in E} [f_{i_j} \leq r] \cap \bigcap_{j \in N \setminus E} [f_{i_j} \geq s] = \emptyset.
\]
(Here $[f \leq r] = \{x : f(x) \leq r\}$ and $[f \geq s] = \{x : f(x) \geq s\}$.) Then the following properties hold:

\noindent (i) There exists a real number $C$ such that for all $x \in X$,
\[
\sum_{i=1}^\infty |f_i(x) - f_{i+1}(x)| \leq C.
\]

\noindent (ii) Suppose further that $X$ is a compact metric space and that the $f_n$'s are continuous. Then the sequence $(f_i)$ converges pointwise to a function $f$ which is $DBSC$ (Difference of Bounded Semi-Continuous functions).
\end{Lemma}

\begin{proof}
(i): First, we consider the case where the functions are discrete-valued. For $n \in \mathbb{N}$, define
\[
f_i^n(x) = \frac{1}{n} \left\lfloor n f_i(x) \right\rfloor \quad \text{for all } i < \omega,
\]
where $\left\lfloor x \right\rfloor$ is the floor function, i.e., the greatest integer less than or equal to $x$. By the assumption, it is straightforward to verify that for all $x \in X$ and $n \in \mathbb{N}$, 
\[
\sum_{i=1}^\infty |f_i^n(x) - f_{i+1}^n(x)| \leq 2N - 2.
\]
This was established for $n = 1$ in Lemma~2.8 of  \cite{K-Baire}  and for an arbitrary $n$, the proof is essentially a reduction to the case $n=1$.
We also prove it for $n = 2$, and for an arbitrary $n$, it is done similarly.\footnote{In fact, the proof for an arbitrary $n$ is similar to the case $n=2$, and we chose this case only for simplicity.}  
The reader can easily verify that the following equality holds for any $x\in X$.\footnote{For easier visualization and intuitive understanding, one can use a diagram that oscillates in a zigzag pattern.}
$$|f_i^2(x)-f_{i+1}^2(x)|=|\min(f_i^2(x),1/2)-\min(f_{i+1}^2(x),1/2)|+$$
$$~~~~~~~~~~~~~~~~~~~~~~~~|\max(f_i^2(x),1/2)-\max(f_{i+1}^2(x),1/2)|.$$
On the other hand, according to the assumptions and using Lemma 2.8 in \cite{K-Baire}, we have:
 $\sum_{k=1}^\infty|\min(f_i^2(x),1/2)-\min(f_{i+1}^2(x),1/2)|\leq(\frac{2N-2}{2})$
 and $\sum_{k=1}^\infty|\max(f_i^2(x),1/2)-\max(f_{i+1}^2(x),1/2)|\leq(\frac{2N-2}{2})$.
This leads to the desired inequality below that we aim to establish.
\[
\sum_{i=1}^\infty |f_i^2(x) - f_{i+1}^2(x)| \leq 2N - 2.
\]

 Therefore, for all $k \in \mathbb{N}$, we have
\[
\sum_{i=1}^k |f_i(x) - f_{i+1}(x)| = \lim_{n \to \infty} \sum_{i=1}^k |f_i^n(x) - f_{i+1}^n(x)| \leq 2N - 2.
\]
Thus, 
\[
\sum_{i=1}^\infty |f_i(x) - f_{i+1}(x)| \leq 2N - 2.
\]

(ii): By Fact~\ref{properties}(i), $f \in DBSC$ if and only if there exists a constant $C > 0$ and a uniformly bounded sequence of continuous functions $(f_i)$ such that $f_i \to f$ pointwise and 
\[
\sum_{i=1}^\infty |f_i(x) - f_{i+1}(x)| \leq C \quad \text{for all } x \in X.
\]
By (i), the latter condition holds, so the limit function $f$ is indeed $DBSC$. (We observe that, given the fact that the independence property is uniformly blocked for $(f_i)$, this sequence clearly converges to a function, which we denote by $f$. This follows from Rosenthal’s theorem (\cite{Ros}); compare the equivalence (i)~$\iff$~(iii) of Theorem~\ref{NIP=Baire-1/2} below.)
\end{proof}

Note that, in contrast to \cite[Lemma~2.8]{K-Baire}, we cannot expect a converse to the direction (i)~$\Rightarrow$~(ii) of Lemma~\ref{bounded variation-continuous}.

\medskip
Let $A$ be a set of parameters in the monster model $\cal U$ and $\phi(x,y)$ a formula. For a parameter $a\in A$, recall that the continuous function $\phi(x,a):S_{\phi}(A)\to[0,1]$ is defined by $p\mapsto\phi(b,a)$ for some (any) $b\models p$.
\begin{Remark} \label{A=U} Consider the function $\phi(x,a):S_\phi({\cal U})\to[0,1]$ as above. For every set $A$ contains $a$, we can replace this function with a function $\phi(x,a):S_\phi(A)\to[0,1]$, which is defined by  $p\mapsto\phi(b,a)$ for some (any) $b\models p$.
In other words, rather than considering the function on $S_\phi({\cal U})$, we may consider it on $S_\phi(A)$ instead.
Indeed, this transformation preserves all the information encoded in the function. For instance, assuming that $A$ is countable, we may exploit the fact that $S_\phi(A)$ is metrizable. (Recall that a compact and second-countable space is metrizable.)
This fact will be utilized in the subsequent developments of this section.
\end{Remark}

We now provide a characterization of $SOP$ in terms of function spaces.

\begin{Theorem} \label{SOP characterization}
Let $T$ be a complete continuous theory and $\mathcal{U}$ its monster model. The following are equivalent:

\noindent (i) $T$ has $SOP$.

\noindent (ii) There exist a formula $\phi(x,y)$ and a sequence $(a_i)$ such that the independence property is uniformly blocked for $\phi(x,y)$ on $(a_i)$, and the sequence 
\[
\left(\phi(x, a_i) : S_\phi(\mathcal{U}) \to [0,1] \right)
\]
converges to a $DBSC$ function which is not continuous.
\end{Theorem}

\begin{proof}
(i)~$\Rightarrow$~(ii): Suppose that $T$ has $SOP$, witnessed by the formula $\phi(x,y)$ and sequences $(b_i),(a_i)$. Then the independence property is uniformly blocked for $\phi(x,y)$ on $(a_i)$, with $N = \{1, 2\}$ and $E = \{2\}$. By Lemma~\ref{bounded variation-continuous} and Remark~\ref{A=U}, the sequence
\[
\left(\phi(x, a_i) : S_\phi(\mathcal{U}) \to [0,1] \right)
\]
converges to a function $f$ that is $DBSC$.\footnote{If $A=\{a_i:i<\omega\}$, then  $X=S_\phi(A)$ is a compact metric space, and we can assume that the functions $\phi(x,a_i)$'s are on $X$.} Since $(b_i), (a_i)$ witness the order property\footnote{If needed, one may employ indiscernible sequences, in a manner analogous to the proof of Theorem~\ref{SOP-main}.}, and by Grothendieck's criterion (Fact~\ref{EG}), the non-order property and continuity are equivalent, it follows that $f$ is not continuous.

(ii)~$\Rightarrow$~(i): This follows from Theorem~\ref{SOP-main}, 
 Remark~\ref{A=U},
 and Grothendieck's criterion (Fact~\ref{EG}).
\end{proof}

Recall that the unit vector basis $(e_n)$ of $c_0$ is defined by
\[
e_n = (0, \ldots, 0, 1, 0, \ldots),
\]
where the 1 occurs in the $n$-th coordinate. The summing basis $(s_n)$ is defined by
\[
s_n = e_1 + e_2 + \cdots + e_n.
\]

\begin{Proposition} \label{SOP->c0}
Let $T$ be a complete theory and $\mathcal{U}$ the monster model. If $T$ has $SOP$, then there exists a formula $\phi(x,y)$ and a sequence $(a_i)$ such that the sequence
\[
\left( \phi(x, a_i) : S_\phi(\mathcal{U}) \to [0,1] \right)
\]
is equivalent to the summing basis for $c_0$.
\end{Proposition}

\begin{proof}
By Theorem~\ref{SOP-main}, there exist $\epsilon > 0$ and a sequence $(a_i, b_i : i < \omega)$ in the monster model $\mathcal{U}$ such that for all $i < j$, we have
\[
\phi(\mathcal{U}, a_i) \leq \phi(\mathcal{U}, a_{i+1}) \quad \text{and} \quad \phi(b_j, a_i) + \epsilon \leq \phi(b_{i+1}, a_j).
\]
Since  $\left( \phi(x, a_i) : S_\phi(\mathcal{U}) \to [0,1] \right)$ is an  increasing sequence of continuous function, its limit is (lower) semi-continuous. Moreover, since the sequence has the order property\footnote{Once more, if needed, one may resort to the use of indiscernible sequences.}, the limit is non-continuous, by Grothendieck's criterion. Now, one can either use Lemma~3.4 in \cite{HOR} or directly prove the desired result. Indeed, it is easy to verify that for every $k$ and scalars $c_1, \ldots, c_k$, we have
\[
\left\| \sum_{i=1}^k c_i \phi(x, a_i) \right\|_\infty \leq \left\| \sum_{i=1}^k c_i s_i \right\|_\infty,
\]
where $(s_i)$ denotes the summing basis of $c_0$. (In all instances, including the application of  \cite[Lemma 3.4]{HOR}, a suitable subsequence may be employed if required.)
\end{proof}

\begin{Remark}
It seems that one cannot expect a converse to Proposition~\ref{SOP->c0}.
 Indeed, by Theorem~1.1 of \cite{Farmaki}, if a sequence $(f_i)$ is equivalent to the summing basis of $c_0$, then its limit is a Baire-1/4 function. Recall that, by Lemma~\ref{bounded variation-continuous} above, in the $SOP$ case, the limit is $DBSC$, and this class is a {\em proper} subclass of Baire-1/4 functions. 

Also, it does not seem that a weaker property than the property (ii) in Theorem~\ref{SOP-main} implies $SOP$. In contrast with $SOP$ (for a theory), we will shortly see that $NIP$ is {\em equivalent to} having no copy of $\ell_1$. Perhaps one of the reasons why Rosenthal's $\ell_1$-theorem \cite{Ros} is more important than his $c_0$-theorem \cite{Ros-c0} is that the former is a {\em first-order property}, while the latter is not.
\end{Remark}

In \cite{K-Baire}, the notion of ``$NSOP$ in a model" was introduced for classical logic. It is natural to generalize it to continuous logic.

\begin{Definition}[$NSOP$ in a model] \label{SOP in model}
Let $T$ be a complete $L$-theory, $\phi(x,y)$ an $L$-formula, and $M$ a model of $T$.

\noindent (i) A set $\{a_{i} : i < \kappa\}$ of $l(y)$-tuples from $M$ is said to be a {\em $SOP$-guarantee for $\phi(x,y)$} if the following two conditions hold simultaneously:
\begin{itemize}
    \item [(1)] The independence property is uniformly blocked for $\phi(x,y)$ on $(a_i : i < \kappa)$.
    \item [(2)] There exist $(b_i : i < \kappa)$ (in the monster model) and real numbers $r < s$ such that $(b_i), (a_i)$ witness the order property with $r < s$.
\end{itemize}

\noindent (ii) Let $A$ be a set of $l(x)$-tuples from $M$. Then $\phi(x,y)$ has a {\em $SOP$-guarantee in $A$} if there exists a countably infinite sequence $(a_{i} : i < \omega)$ of elements of $A$ which is a $SOP$-guarantee for $\phi(x,y)$.

\noindent (iii) Let $A$ be a set of $l(x)$-tuples in $M$. We say that $\phi(x,y)$  {\em rejects a $SOP$-guarantee in $A$} if it does not have a $SOP$-guarantee in $A$.

\noindent (iv) $\phi(x,y)$  {\em rejects a  $SOP$-guarantee in $M$} if it rejects a  $SOP$-guarantee in the set of $l(x)$-tuples from $M$.
\end{Definition}

The following is Remark~2.13 from \cite{K-Baire} for classical first-order logic, which also holds for continuous logic. For the sake of completeness, we present it here.

\begin{Remark} \label{remark sop in set}
Let $T$ be a complete (continuous) $L$-theory, $\phi(x,y)$ an $L$-formula, $M$ a model of $T$, and $A$ a subset of $M$.

\noindent (i) If $\phi$ has a $SOP$-guarantee in $A$, then a Boolean combination of instances of $\phi$ has $SOP$ for the theory $T$. Of course, if $\phi$ has $SOP$ for $T$, then it has a $SOP$-guarantee in some model of $T$.

\noindent (ii) $\phi$ has $NSOP$ for the theory $T$ if and only if it rejects a $SOP$-guarantee in every model $M$ of $T$, and equivalently, it rejects  a $SOP$-guarantee in some model $M$ of $T$ in which all types over the empty set in countably many variables are realized.

\noindent (iii) If $\phi(x,y)$ has a $SOP$-guarantee in some model $M$ of $T$, then there are arbitrarily long $SOP$-guarantees for $\phi$ (of course, in different models).
\end{Remark}

\section{Independence property} In this section we give some
new characterizations  of $NIP$ which are different from the
classical case.

Recall that the definition of $NIP$ is standard in continuous logic and is mentioned in \cite{Ben-Vapnik} and \cite{Ben2}. It is easy to verify that the following definition is equivalent to the standard one.
\begin{Definition} {\em Let $T$ be a complete (continuous) theory and $\phi(x,y)$ a formula.

\medskip\noindent
(i) We say  $\phi(x,y)$  has the {\em  independence property}
($IP$) if  there are real numbers $r<s$ and a sequence
$(a_i:i<\omega)$ in the monster model $\mathcal U$ such that for
all disjoint finite sets $E,F$ the following holds
$$\exists y\big( \bigwedge_{i\in E}\phi(a_i,y)\leqslant r~\wedge~\bigwedge_{i\in F}\phi(a_i,y)\geqslant s\big)$$

\medskip\noindent
(ii) We say  $\phi(x,y)$  has the non independence property
($NIP$) if it has not $IP$.

\medskip\noindent
(iii) We say  $T$ has $NIP$ ($IP$) if every (some) formula has
$NIP$ ($IP$), respectively.}
\end{Definition}

The following definition is exactly equivalent to $NIP$, but we state it separately for emphasis and due to its specific utility.
\begin{Definition} \label{NIP-remark} Let $T$ be a complete (continuous) theory,
$\phi(x,y)$ a formula. We say the {\em independence property is
semi-uniformly blocked for $\phi(x,y)$}  if

for each $r<s$ there is a natural number $N_{r,s}$ and a set
$E\subset\{1,\ldots,N_{r,s}\}$ such that  for every sequence
$(a_i)$ (in the monster model of $T$) and  each $i_1<\cdots<i_{N_{r,s}}<\omega$, the following does not hold
$$\exists y\big(\bigwedge_{j\in E}\phi(a_{i_j},y)\leq r~\wedge~\bigwedge_{j\in N\setminus E}\phi(a_{i_j},y)\geq s\big).$$
\end{Definition}

\begin{Remark}
It is easy to see that $\phi(x,y)$ has $NIP$ if and only if the independence
property is semi-uniformly blocked for $\phi(x,y)$ (see also
Theorem~\ref{NIP=Baire-1/2} below). Note that, as before mentioned
in Remark~\ref{main-remark}, the condition (ii)(1) of
Theorem~\ref{SOP-main} is stronger than $NIP$ in the sense that
there is a natural number $N$ such that for all $r<s$,
$N_{r,s}=N$.
\end{Remark}
In the following, we present a stronger argument and compare some of the above concepts with each other.
\begin{Fact}
Let $T$ be a complete $L$-theory, $\phi(x,y)$ an $L$-formula. Suppose condition (i) of  Lemma~\ref{bounded variation-continuous}  holds, that is, there is a number $C$ such that $\sum_1^\infty|\phi(a_i,b)-\phi(a_{i+1},b)|\leq C$ for all $(a_i),b$ in the monster model. Then $\phi(x,y)$ has $NIP$.
\end{Fact}
\begin{proof} If not, there are $r<s$ such that for any arbitrarily large natural number $N$ there are parameters $a_1,\ldots,a_N,b$ such that $\phi(a_i,b_N)\leq r$ if $i$ is even and $\phi(a_i,b_N)\geq s$ otherwise. Therefore, $\sum_1^N|\phi(a_i,b_N)-\phi(a_{i+1},b_N)|\geq N-1$.
Since $N$ can grow arbitrarily large, by the compactness theorem,
there are $(a_i),b$ such that $\sum_1^\infty|\phi(a_i,b)-\phi(a_{i+1},b)|=\infty$, a contradiction.
\end{proof}
To summarize uniform $IP$ blocking imples the condition (i) in Lemma~\ref{bounded variation-continuous}   and the latter condition implies $NIP$(= semi-uniform blocking).
Using Ramsey's theorem, it is not difficult to show that uniform $IP$ blocking is {\em strictly} stronger than condition (i) in Lemma~\ref{bounded variation-continuous}.  To do this, it suffices to decrease $|r-s|$ and increase the length of the sequence in such a way that the resulting sum remains small. Now, since the length of the sequence can be increased arbitrarily, by the finite Ramsey theorem, one can find a (sufficiently large) subsequence that is shattered (that is, there is no $N$ that satisfies the definition of uniform $IP$ blocking). See also the direction (i)~$\Rightarrow
$~(ii) of Lemma~2.8 in \cite{K-Baire}.\footnote{For this purpose, one should switch the roles of $r$ and $s$ with those of 0 and 1 in that context, and assume that $r$ and $s$ are sufficiently close to each other.}

\subsection{Characterization of $NIP$}  
We first present a characterization of $NIP$ in terms of function spaces. Notably, the conditions (iv) and (v) appear to be new to model theorists.

\begin{Theorem}[Characterization of $NIP$] \label{NIP=Baire-1/2} 
Let $T$ be a complete $L$-theory, $\phi(x,y)$ an $L$-formula, and $\mathcal{U}$ the monster model of $T$. Then the following are equivalent:

\begin{enumerate}
    \item[(i)] $\phi(x,y)$ has $NIP$ for $T$.
    
    \item[(ii)] The independence property is semi-uniformly blocked for $\phi(x,y)$.
    
    \item[(iii)] For any sequence (not necessarily indiscernible) $(a_i : i < \omega)$, there exists a subsequence $(a_{j_i} : i < \omega)$ such that, for each $b \in \mathcal{U}$, the sequence $\phi(a_{j_i}, b)$ converges.
    
    \item[(iv)] For any sequence (not necessarily indiscernible) $(a_i : i < \omega)$, there exists a subsequence $(a_{j_i} : i < \omega)$ such that the sequence $\phi(a_{j_i}, y)$ converges to a function $f$ that is Baire-1/2.
    
    \item[(v)] There is no sequence $(a_i)$ such that the sequence $(\phi(a_i, y) : S_{\tilde{\phi}}(\mathcal{U}) \to [0,1])$ is equivalent in the supremum norm to the usual basis of $\ell_1$.
\end{enumerate}
\end{Theorem}

\begin{proof}
The implications (ii)~$\Rightarrow$~(i) and (iv)~$\Rightarrow$~(iii) are evident.

\medskip
(i)~$\Rightarrow$~(ii): This follows directly from the compactness theorem.

\medskip
The equivalence (i)~$\Leftrightarrow$~(iii) is well-known and considered folklore. (See the equivalence (i)~$\Leftrightarrow$~(v) of Lemma~3.12 in \cite{K3}.)

\medskip
The equivalence (iii)~$\Leftrightarrow$~(v) is a consequence of one of the most elegant results in Banach space theory, Rosenthal's $\ell_1$-theorem \cite[Theorem~1]{Ros}.

\medskip
(i)~$\Rightarrow$~(iv): Since (i) and (iii) are equivalent, for any sequence $(a_i:i<\omega)$ there exists a subsequence $(a_{j_i} : i < \omega)$ such that the sequence $\phi(a_{j_i}, y)$ converges to a function $f$. It is sufficient to show that $f$ is Baire-1/2.

Now, suppose for a  contradiction that $f$ is not Baire-1/2.
We shall make use of Proposition 2.3 from \cite{HOR}.
For this purpose, we use the ordinal indices $K_m(f;r,s)$ and $\alpha(f;r,s)$, which are defined on page 7  of that paper.\footnote{For a function $F$ on a compact metric space $K$ and real numbers $a<b$, the index   $K_\alpha(F;a,b)$ in \cite{HOR} is essentially the Bourgain rank, whose idea also appears in \cite{BFT}, although the recent paper does not mention the rank (index)  explicitly. 
To keep the proof concise and avoid complexity, we do not provide the definition of these indices and instead refer to \cite[page 7]{HOR}.}

 By the equivalence (1)~$\Leftrightarrow$~(5) in Proposition~2.3 of \cite{HOR}, there exist $r < s$ such that the ordinal index $\alpha(f; r, s)$ is infinite. Equivalently, for all natural numbers $m$, we have $K_m(f; r, s) \neq \emptyset$. 

By Lemma~3.1 in \cite{HOR}, for every $m$, and for $r < r' < s' < s$, there exists a subsequence $(\phi(a_{j_i}, y) : i < \omega)$ such that the following holds:
\[
\exists y \bigg( \bigwedge_{i \in E} \phi(a_i, y) \leq r' ~\wedge~ \bigwedge_{i \in F} \phi(a_i, y) \geq s' \bigg),
\]
for all disjoint subsets $E, F$ of $\{1, \ldots, m\}$.

By the compactness theorem, this implies that $\phi$ has the independence property ($IP$), which contradicts the assumption that $\phi$ has $NIP$. (Note that $r'$ and $s'$ are fixed, while $m$ can grow arbitrarily large. Therefore, the compactness theorem applies here.)
\end{proof}

\begin{Remark} 
In general, it seems one cannot expect the limits in Theorem~\ref{NIP=Baire-1/2}(iv) to belong to a proper subclass of Baire-1/2. Indeed, using Theorem~6.1 of \cite{HOR}, ine might be able to demonstrate that there exist a theory $T$, an $NIP$ formula $\phi(x,y)$, and a sequence $(a_i)$ as described, such that the limit is not Baire-1/4. 

\noindent (Recall that $DBSC \subsetneqq \text{Baire-1/4} \subsetneqq \text{Baire-1/2}$.)
\end{Remark}

\noindent For the sake of completeness, we state the result separately for classical ($\{0,1\}$-valued) logic. The following was proved in \cite[Proposition~2.14]{K-Baire}, although the condition (v) is new.

\begin{Corollary}[Classical case] \label{NIP=DBSC}
Let $T$ be a complete classical $L$-theory, $\phi(x,y)$ an $L$-formula, and $\mathcal{U}$ the monster model of $T$. Then the following are equivalent:

\begin{enumerate}
    \item[(i)] $\phi(x,y)$ has $NIP$ for $T$.
    
    \item[(ii)] For any sequence (not necessarily indiscernible) $(a_i : i < \omega)$, the independence property is uniformly blocked for $\phi(x,y)$ on $(a_i : i < \omega)$. (See Definition~\ref{uniformly blocked}.)
    
    \item[(iii)] For any sequence (not necessarily indiscernible) $(a_i : i < \omega)$, there exists a subsequence $(a_{j_i} : i < \omega)$ such that, for each $b \in \mathcal{U}$, the sequence $\phi(a_{j_i}, b)$ converges.
    
    \item[(iv)] For any sequence (not necessarily indiscernible) $(a_i : i < \omega)$, there exists a subsequence $(a_{j_i} : i < \omega)$ such that the sequence $\phi(a_{j_i}, y)$ converges to a function $f$ that belongs to $DBSC$.
    
    \item[(v)] There is no sequence $(a_i)$ such that the sequence $(\phi(a_i, y) : S_{\tilde{\phi}}(\mathcal{U}) \to [0,1])$ is equivalent in the supremum norm to the usual basis of $\ell_1$.
\end{enumerate}
\end{Corollary}

\begin{proof}
The result follows immediately from Theorem~\ref{NIP=Baire-1/2}, as for simple functions, we have $DBSC = \text{Baire-1/2}$ (Fact~\ref{properties}(iv)).
\end{proof}

\subsection{Baire-1/2 Definability of Coheirs} 

In \cite{K3}, the notion of ``$NIP$ in a model" (within the framework of continuous logic), as defined in Definition~\ref{NIP in model} below, was introduced. Applications of this concept, such as the Baire-1 definability of coheirs over $NIP$ models, were also presented. Subsequently, \cite{KP} extended the applications of this notion to the framework of classical logic. 

In this section, we establish a stronger result under a stronger assumption. Specifically, we replace ``$NIP$ in a model" with ``$NIP$ for a theory." Our goal is to demonstrate that coheirs in $NIP$ theories are Baire-1/2 definable (cf. Definition~\ref{Baire-1/2 definable} below).

 Let $M$ be a separable model, $M^*$ a saturated elementary extension of $M$, $X = S_{\tilde{\phi}}(M)$, and 
$A = \{\phi(a, y) : X \to [0, 1] \mid a \in M\}$. A type $p(x) \in S_\phi(M^*)$ is called a {\em coheir (of a type) over $M$} if it is {\em approximately finitely satisfiable in $M$}; that is, for every condition $\varphi = 0$ in $p(x)$ (i.e., $\varphi(p) = 0$) and every $\epsilon > 0$, the condition $|\varphi| \leq \epsilon$ is satisfiable in $M$. 

Clearly, in the classical case (i.e., when $\epsilon = 0$ and $(\varphi = 0) \in p$ means $\varphi \in p$), this notion coincides with the usual notion of a coheir. Alternatively, using the notation introduced in Section~2.2, $p(x) \in S_\phi(M^*)$ can be regarded as a function from $M^*$ to $[0, 1]$, defined by $b \mapsto \phi(p, b)$. Then, $p$ is a coheir of a type over $M$ if, for each $\epsilon > 0$ and each $b_1, \ldots, b_n \in M^*$, there exists some $a \in M$ such that $|\phi(p, b_i) - \phi(a, b_i)| < \epsilon$ for all $i \leq n$.

Let $p(x)$ be a coheir of a type over $M$, and let $b_1, b_2 \in M^*$. If $b_1$ and $b_2$ have the same $\tilde{\phi}$-type over $M$, then $\phi(p, b_1) = \phi(p, b_2)$. Indeed, since $p(x)$ is approximately finitely satisfiable in $M$, for any $\epsilon > 0$, there exists some $a \in M$ such that $|\phi(p, b_i) - \phi(a, b_i)| < \epsilon$ for $i = 1, 2$. Therefore,
\[
|\phi(p, b_1) - \phi(p, b_2)| \leq |\phi(p, b_1) - \phi(a, b_1)| + |\phi(a, b_1) - \phi(a, b_2)| + |\phi(p, b_2) - \phi(a, b_2)| < 2\epsilon.
\]
Since $\epsilon$ is arbitrary, it follows that $\phi(p, b_1) = \phi(p, b_2)$.

For simplicity, for all $b \in M^*$, we will write $\phi(p, b) = \phi(p, q)$, where $q = \text{tp}_{\tilde{\phi}}(b/M) \in X$. To summarize, $p(x)$ is a function in $[0, 1]^X$, defined by $q \mapsto \phi(p, q)$ as described above.

  By Remark~3.19 of \cite{K3}, similar to the classical logic, there is a correspondence between the coheirs $p(x)$ of types over $M$ and the functions in the pointwise closure of $A$ (see also the explanation before Definition~3.18 of \cite{K3}). For the sake of completeness, we present it and provide a proof.

\begin{Fact}\footnote{This fact was first stated by Pillay in \cite{Pillay-Grothendieck} for classical logic.}\label{Pillay remark}
Let $X$ be as above and $f \in [0, 1]^X$. Then $f$ belongs to the pointwise closure of $A$ if and only if there is a $p(x) \in S_{\phi}(M^*)$ such that $p(x)$ is approximately finitely satisfiable in $M$ and $f(q) = \phi(p, q)$ for all $q \in X$.
\end{Fact}

\begin{proof}
Recall that a typical neighborhood of $f$ is of the form 
\[
U_f(q_1, \ldots, q_n; \epsilon) = \{g \in [0, 1]^X : |f(q_i) - g(q_i)| < \epsilon \text{ for } i \leq n \},
\]
where $\epsilon > 0$ and $\{q_1, \ldots, q_n\}$ is a finite subset of $X$. Thus, $f$ belongs to the pointwise closure of $A$ if and only if for each neighborhood $U_f(q_1, \ldots, q_n; \epsilon)$, there exists $\phi(a, y) \in A$ (i.e., $a \in M$) such that $|f(q_i) - \phi(a, q_i)| < \epsilon$ for all $i \leq n$. 

This is equivalent to saying that there exists $p(x) \in S_{\phi}(M^*)$ that is approximately finitely satisfiable in $M$, where $\phi(p, q) = f(q)$ for all $q \in X$. In this case, recall that $\phi(x, q) = r$ belongs to $p(x)$ (or $\phi(p, q) = r$) if and only if $\phi(p, b) = r$ for all $b \in M^*$ such that $q = \text{tp}_{\tilde{\phi}}(b / M) \in X$.
\end{proof}

\begin{Definition}[Baire-1/2 definability] \label{Baire-1/2 definable}
  {\em Let $M$ be a separable model, $M^*$ a saturated elementary extension of it, and $\phi(x,y)$ a formula. Let $p(x) \in S_\phi(M^*)$ be approximately finitely satisfiable in $M$. We say that $p$ is {\em Baire-1/2 (resp. $DBSC$) definable over $M$} if there is a sequence $\phi(a_n,y)$, with $a_n \in M$, such that the sequence $(\phi(a_n,y))$ of functions on $S_{\tilde{\phi}}(M)$ converges pointwise to a function $f$ which is Baire-1/2 (resp. $DBSC$) and $\phi(p,q) = f(q)$ for all $q \in S_{\tilde{\phi}}(M)$.}
\end{Definition}

\begin{Remark}
(i) Recall that in the above definition, as $p(x)$ is approximately finitely satisfiable in $M$, if $b_1, b_2 \in M^*$ have the same $\tilde\phi$-type over $M$, then $\phi(p,b_1) = f(tp(b/M)) = \phi(p,b_2)$.

(ii) Note that, as $f$ is a Baire-1 function, for every open set $O \subseteq \mathbb{R}$, the set $f^{-1}(O)$ is $F_\sigma$, i.e., it is a countable union of closed sets. Of course, Baire-1/2 definability says more because Baire-1/2 $\subsetneq$ Baire-1. Indeed, in classical logic, $f^{-1}(O)$ is a disjoint difference of closed sets $W_1, \ldots, W_m$. \footnote{Note that such functions are Baire-1 but not conversely. See \cite[Proposition~2.2]{CMR} and \cite[Remark~2.15]{K-Baire}.}

(iii) Note that Baire-1/2 (or $DBSC$) definability is a generalization of the usual definability. That is, a type is called definable if the function $f$ above is {\em continuous}. Pillay \cite{Pillay-Grothendieck} showed that ``stability in a model" is equivalent to the definability of ``coheirs".\footnote{Unfortunately, some authors have mistakenly claimed that ``stability in a model" is equivalent to the definability of ``types", and thus made other false claims. The reason for this is probably that they were not aware of Fact~\ref{Pillay remark} above.} In \cite{KP}, this result was generalized to ``$NIP$ in a model". In \cite{K-Baire}, the connection between $NIP$ for classical theories and $DBSC$ functions was studied, although the equivalence of $NIP$ (for theories) and $DBSC$ definability of coheirs was not studied. We do this work in the present paper, moreover for continuous logic.
\end{Remark}

 Because we will refer to a proof for the notion `$NIP$ in a model' (\cite[Definition~3.11]{K3}), we recall this notion:

\begin{Definition}[$NIP$ in a model] \label{NIP in model}
  {\em Let $M$ be a model, and $\phi(x,y)$ a formula. We say that
  {\em $\phi(x,y)$ has $NIP$ in $M$} if for each sequence
  $(a_n)\subseteq M$, and $r > s$, there are some {\em finite} disjoint subsets $E, F \subseteq \mathbb{N}$ such that the following does not hold:
  $$\exists y \left( \bigwedge_{i \in E} \phi(a_i, y) \leq r \wedge \bigwedge_{i \in F} \phi(a_i, y) \geq s \right).$$}
\end{Definition}

In Lemma~3.12 of \cite{K3}, there is a long list of equivalences for this notion. The only additional thing we need to remember is that, due to the compactness theorem,  $NIP$ (for a theory) is stronger than $NIP$ in a model. We are now ready to give the definability result. The proof uses Theorem~\ref{NIP=Baire-1/2} and some crucial results from \cite{BFT}.

\begin{Theorem} \label{NIP definability}
  Let $T$ be a theory and $\phi(x,y)$ a formula. The following are equivalent:

  \noindent (i) $\phi(x,y)$ has $NIP$.

  \noindent (ii) For any separable model $M$ and saturated elementary extension $M^*$ of it, whenever $p(x) \in S_\phi(M^*)$ is approximately finitely satisfiable in $M$, it is Baire-1/2 definable over $M$.

  \noindent (iii) For any separable model $M$ and saturated elementary extension $M^*$ of it, the number of $p(x) \in S_\phi(M^*)$ which are approximately finitely satisfiable in $M$ is $< 2^{2^{\aleph_0}}$.
\end{Theorem}

\begin{proof} (i)~$\Leftrightarrow$~(ii): 
  As $\phi$ has $NIP$ (for the theory) iff for every separable model $M$ it has $NIP$ in $M$, the equivalence (i)~$\Leftrightarrow$~(ii) is a straightforward adaptation of the argument from Proposition~2.6 of \cite{K-definable} to continuous logic. Although there are some considerations. Indeed, recall that for a separable model $M$, the space $X = S_{\tilde{\phi}}(M)$ is compact and Polish. Let $M_0$ be a countable dense subset of $M$ and $A_0 = \{\phi(a, y) : X \to [0, 1] : a \in M_0\}$. Since $\phi(x, y)$ has a modulus of uniform continuity (cf. \cite[Theorem~3.5]{BBHU}) and the set $X_0 = \{q \in X : q = tp_{\tilde\phi}(b/M) \text{ for some } b \in M\}$ is dense in $X$, it is easy to check that $A_0$ and $A = \{\phi(a, y) : X \to [0, 1] : a \in M\}$ have the same pointwise closure. Thus, we can work with the countable set $A_0$, and the countability assumption of Fact~\ref{BFT criterion} is fulfilled throughout the proof. By Fact~\ref{Pillay remark}, any type $p(x) \in S_\phi(M^*)$ is a coheir iff there is a function $f$ in the closure of $A$ such that $f(q) = \phi(p, q)$ for all $q \in X$. By the equivalences (1)~$\Leftrightarrow$~(2)~$\Leftrightarrow$~(3) of Fact~\ref{BFT criterion}, $\phi$ has $NIP$ in $M$ iff the closure of $A$ in $[0, 1]^X$ is angelic, iff every $f$ in the closure of $A$ is Baire-1. The only additional point is that by Theorem~\ref{NIP=Baire-1/2}(iv), the function $f$ in the pointwise closure of $A$ that defines $p(x)$ is Baire-1/2.

  (ii)~$\Rightarrow$~(iii) is evident, since every Baire-1/2 function is the limit of a sequence of continuous functions.

  (iii)~$\Rightarrow$~(i) is the usual proof for many coheirs of formulas with the independence property. Namely, suppose that (i) fails; that is, there are $r < s$ and a sequence $(a_n)$ in a separable model $M$ such that:
  $$D(I) = \{ b \in S_{\tilde{\phi}}(M) : \bigwedge_{n \in I} \phi(a_n, b) \leq s \wedge \bigwedge_{n \in \mathbb{N} \setminus I} \phi(a_n, b) \geq r \} \neq \emptyset \quad \text{for every} \quad I \subseteq \mathbb{N}.$$
  If ${\cal F}$ and $\mathcal{G}$ are distinct ultrafilters on $\mathbb{N}$, there is an $I \subseteq \mathbb{N}$ such that $I \in \mathcal{F}$ and $\mathbb{N} \setminus I \in \mathcal{G}$, so that $\lim_{\mathcal{F}} \phi(a_n, b) \leq s < r \leq \lim_{\mathcal{G}} \phi(a_n, b)$ for every $b \in D(I)$. Since there are $2^{2^{\aleph_0}}$ distinct ultrafilters on $\mathbb{N}$, $(\phi(a_n, y) : S_{\tilde{\phi}}(M) \to [0, 1])$ has $2^{2^{\aleph_0}}$ distinct cluster points, and (iii) fails.
\end{proof}


We state the result separately for the classical ($\{0,1\}$-valued) logic.

\begin{Corollary}[Classical case] \label{NIP definability classic}
  Let $T$ be a classical theory and $\phi(x,y)$ a formula. The following are equivalent:

  \noindent (i) $\phi(x,y)$ has $NIP$.

  \noindent (ii) For any countable model $M$ and saturated elementary extension $M^*$ of it, whenever $p(x) \in S_\phi(M^*)$ is finitely satisfiable in $M$, then it is $DBSC$ definable over $M$.

  \noindent (iii) For any countable model $M$ and saturated elementary extension $M^*$ of it, the number of $p(x) \in S_\phi(M^*)$ which is finitely satisfiable in $M$ is $< 2^{2^{\aleph_0}}$.
\end{Corollary}

\begin{proof}
  Immediate, by Theorem~\ref{NIP definability} and the fact that for simple functions, $DBSC$ = Baire-1/2 (cf. Fact~\ref{properties}(iv)).
\end{proof}

Of course, the direction (i)~$\Rightarrow$~(ii) can be considered as a consequence of Proposition~2.6 of \cite{HP}. Indeed, since assuming $NIP$, every coheir of a type over $M$ is {\em strongly Borel definable} in the sense of \cite{HP}, the function that defines the coheir is the indicator function of a finite Boolean combination of closed sets over $M$. It is easy to check that such a function is $DBSC$, by Proposition~2.2 of \cite{CMR}. (See also Remark~2.15 of \cite{K-Baire}.)

\begin{Remark}  
Let $T$ be a classical theory, $M$ a countable model of $T$, and  $\phi(x,y)$ a $NIP$ formula. It can be proved that every local Keisler measure which is finitely satisfiable in $M$ is Baire-1/2 definable. Indeed, Ben Yaacov and Keisler \cite[Corollary 2.10]{Ben-Keisler} proved that every measure in the theory $T$ corresponds to a type in the randomization $T^R$ of $T$. On the other hand, the randomization of every $NIP$ theory is a $NIP$ continuous theory (see \cite{Ben-Vapnik}). (In fact, the argument of this result is local (formula-by-formula), and one can easily check that the corresponding formula $\phi^R$ of $\phi$ in the randomization is $NIP$ if $\phi$ is $NIP$. See also \cite{KP-random}.) By Theorem~\ref{NIP definability} above, every $\phi^R$-type in $T^R$ is Baire-1/2 definable, and so its corresponding measure in $T$ is Baire-1/2 definable. This will be discussed in detail in a future work.
\end{Remark}

\bigskip\noindent
{\bf History.} In \cite{HPP} and \cite{HP}, the Borel definability of coheirs and the strongly Borel definability of invariant types in $NIP$ theories were proved, respectively. The notion ``$NIP$ in a model" was introduced in \cite{K3} and some variants of definability of types and coheirs were proved. The correspondence between Shelah's theorem (in classical logic) and the Eberlein--\v{S}mulian theorem was also pointed out in \cite{K3}. This approach was later followed in \cite{K-definable} and \cite{KP}, and more applications for the notion ``$NIP$ in a model" were presented. In \cite{K-Baire}, the usefulness of this approach for the study of model theoretic properties of ``theories" (in classical logic) was demonstrated. In the present paper, this utility is extended to continuous logic, and new results are presented for both classical logic and continuous logic. In the next section, we examine the effectiveness of this approach in the classification of ``continuous theories."

\section{Discussion and Thesis}  \label{discussion}
In this section, we argue why there is no result similar to Shelah's theorem for classical logic in Banach space theory; equivalently, there is no such theorem in real-valued function theory or continuous logic. We explain the strategy for finding a counterexample and give a new classification of continuous theories in terms of function spaces. Although the following is mainly expository, it is (in our view) very illuminating. (This section can't be read without a firm grasp of \cite[Section~3]{K-Baire}.)

\subsection{Discussion}
In \cite{K5}, a `weak' form of Shelah's theorem for continuous logic was proved. Although, as mentioned above, we argue that the exact form of Shelah's theorem ($OP$ $\Leftrightarrow$ $IP$ or $SOP$) does not hold in continuous logic. The reader can compare the proofs of \cite[Proposition~1.9]{K5} and Theorem~\ref{SOP-main} above to see where the key difference lies. In fact, the notion $SCP$ (or $NwSOP$) in \cite{K5} is strictly stronger than the notion $SOP$ (cf. Definition~\ref{SOP-definition} above), and the usual argument of Shelah's theorem does work with the latter notion. On the other hand, given the above, the following table is available:

$$
\begin{tabular}{|c|}
  \hline
  Continuous Logic \\
  \hline
\end{tabular}
$$
$$
\begin{tabular}{|c|c|}
  \hline
  Stability & $C(X)$  \\
  \hline
  $NIP$ & $B_1(X)\setminus B_{1/2}(X)=\emptyset$  \\
  \hline
  $NSOP$ & $DBSC(X)\setminus C(X)=\emptyset$  \\
  \hline
\end{tabular}
$$
(Recall that the correspondence between stability and $C(X)$ follows from the Eberlein--\v{S}mulian Theorem. The second and third correspondences follow from Theorems~\ref{NIP=Baire-1/2} and \ref{SOP characterization}, respectively.)

\medskip
In the classical case, from \cite{K-Baire}, we have the following table:

$$
\begin{tabular}{|c|}
  \hline
  Classical Logic \\
  \hline
\end{tabular}
$$
$$
\begin{tabular}{|c|c|}
  \hline
  Stability & $C(X)$  \\
  \hline
  $NIP$ & $B_1(X)\setminus DBSC(X)=\emptyset$  \\
  \hline
  $NSOP$ & $DBSC(X)\setminus C(X)=\emptyset$  \\
  \hline
\end{tabular}
$$
(The correspondences in the classical case follow from the Eberlein--\v{S}mulian Theorem, Propositions~2.13 and Remark~2.10 of \cite{K-Baire}, respectively. See also Section~3 of \cite{K-Baire}, especially Remark~3.4.)

\medskip
Recall from Fact~\ref{properties}(iv) that for simple functions, $B_{1/2}(X) = DBSC(X)$. Of course, as mentioned in Fact~\ref{properties}(iii), $B_{1/2}(X) \neq DBSC(X)$ in real-valued functions, i.e., the continuous case. We argue that this is the reason why Shelah's theorem holds in classical logic and no one should expect something similar to be true in continuous logic. To complete the discussion, we recall some results from function theory and Banach space theory.

The main result of \cite{HOR}, i.e., Theorem~B, is very similar to Shelah's theorem, but for the same reason, it is not a perfect analog of the latter. Here we explain this similarity. First, we recall two notions from Banach space theory. Let $(x_n)$ be a seminormalized basic sequence in a Banach space $Y$. A basic sequence $(e_n)$ is said to be a {\em spreading model} of $(x_n)$ (or $Y$) if for all natural numbers $k$ and $\epsilon > 0$, there exists $N$ such that for any $N < n_1 < \cdots < n_k$ and $(r_i)_1^k \subseteq \mathbb{R}$ with $\sup_i |r_i| \leq 1$, we have
$$
\left|\left\|\sum_{i=1}^k r_i x_{n_i}\right\| - \left\|\sum_{i=1}^k r_i e_i\right\|\right| < \epsilon.
$$
Roughly speaking, $(e_n)$ is a {\em Morley sequence} in an elementary extension of $Y$ in the sense of model theory. Finally, recall that a sequence $(g_n)$ in a Banach space is a {\em convex block subsequence} of $(f_n)$ if $g_n = \sum_{i=p_n+1}^{p_{n+1}} r_i f_i$, where $(p_n)$ is an increasing sequence of integers, $(r_i) \in \mathbb{R}^+$, and for each $n$, we have $\sum_{i=p_n+1}^{p_{n+1}} r_i = 1$. Again, roughly speaking, a convex block subsequence of $(f_n)$ can be considered as a sequence of Boolean combinations of the instances of $f_n$. In the following theorem, we can assume that $X$ is the type space $S_{\tilde\phi}(M)$ where $\phi(x,y)$ is a formula, $M$ is a separable model of a continuous theory, and $(f_n)$ is a sequence of the form $(\phi(a_n,y): n < \omega)$ where $a_n \in M$ and $\phi(a_n, y)$ is a continuous function on $X$ as previously defined.

\medskip
\noindent{\bf Theorem~$\mathcal{B}$:} Let $X$ be a compact metric space, $f \in B_1(X) \setminus C(X)$, and $(f_n)$ be a uniformly bounded sequence in $C(X)$ which converges to $f$.

\begin{itemize}
  \item [(a)]
    \begin{itemize}
      \item [(i)] If $f \notin B_{1/2}(X)$, then $(f_n)$ has a subsequence whose spreading model is equivalent to the unit vector basis of $\ell_1$.
      \item [(ii)] If every convex block basis of $(f_n)$ has a spreading model equivalent to the unit vector basis of $\ell_1$, then $f \notin B_{1/2}(X)$.
    \end{itemize}
  \item [(b)]
    \begin{itemize}
      \item [(i)] If $f \in B_{1/4}(X)$, then some convex block basis of $(f_n)$ has a spreading model equivalent to the summing basis of $c_0$.
      \item [(ii)] If $(f_n)$ has a spreading model equivalent to the summing basis of $c_0$, then $f \in B_{1/4}(X)$.
    \end{itemize}
\end{itemize}

\medskip
\noindent {\em Explanation.} Theorem~$\mathcal{B}$(a) and Theorem~$\mathcal{B}$(b) are Theorems~II.3.5 and II.3.6 of \cite[Chapter~23]{Rosenthal-handbook}. Theorem~$\mathcal{B}$(a)(i), Theorem~$\mathcal{B}$(a)(ii), and Theorem~$\mathcal{B}$(b)(i) follow from \cite[Theorem~B(a)]{HOR}, \cite[Theorems~3.7]{HOR}, and \cite[Theorem~B(b)]{HOR}, respectively. Theorem~$\mathcal{B}$(b)(ii) is due to V. Farmaki \cite[Theorem~1.1]{Farmaki} answering an open question raised in \cite{HOR}.

\medskip
Notice that Theorem~$\mathcal{B}$(a)(ii) is a converse to Theorem~$\mathcal{B}$(a)(i) and Theorem~$\mathcal{B}$(b)(ii) is a converse to Theorem~$\mathcal{B}$(b)(i). Theorem~$\mathcal{B}$(a) should be compared with Theorem~\ref{NIP=Baire-1/2}(v) above. In other words, it corresponds to the independence property for a {\em theory}. On the other hand, Theorem~$\mathcal{B}$(b) should be compared with Proposition~\ref{SOP->c0} above. Therefore, it is related to the strict order property for a {\em theory}. The reason that Theorem~$\mathcal{B}$(a) and Theorem~$\mathcal{B}$(b) cannot interact and join together to form a single theorem similar to Shelah's theorem (or even the Eberlein--\v{S}mulian Theorem) is that $B_{1/4}(X) \neq B_{1/2}(X)$ in general (cf. Fact~\ref{properties}(iii)). (Moreover, recall from \cite{OS} that there exists a separable Banach space $X$ so that no spreading model of $X$ contains $c_0$ or $\ell_p$ ($1 \leq p < \infty$).) This is precisely the reason why one should not expect the exact form of Shelah's theorem to be established for continuous logic. This is not the case in classical logic because for $\{0,1\}$-valued (even simple) functions, $DBSC(X) = B_{1/4}(X) = B_{1/2}(X)$. It seems that the {\em exotic} (or {\em unnatural}) examples in Banach space theory, those that do not contain an infinite-dimensional reflexive subspace or an isomorph of $\ell_1$ or $c_0$, come from this point.

\medskip
In the following, we argue that finding a counterexample to a perfect analog of Shelah's theorem for continuous logic is related to a specific type of Banach spaces. In \cite[Section~4.2]{K4}, we assigned a Banach space $V_{\phi,M}$ to each formula $\phi$ and model $M$ as follows. Let $T$ be a (continuous) theory, $M$ a model of it, $\phi(x,y)$ a formula, and $X = S_{\tilde\phi}(M)$ the space of complete $\tilde\phi$-types over  $M$. The {\em space of linear $\tilde\phi$-definable relations on $M$}, denoted by $V_{\phi,M}$, is the closed subspace of $C(X)$ generated by the set $\{\phi(a,y): X \to [0,1] \mid a \in M\}$. (Notice that each $\phi(a,y): X \to [0,1]$, with $a \in M$, is continuous and $V_{\phi,M}$ is a Banach space with the sup-norm.) It is easy to show that:
\begin{itemize}
  \item [(i)] A (continuous) theory is stable if and only if for each separable model $M$ and each formula $\phi$, the Banach space $V_{\phi,M}$ is reflexive (cf. \cite[Corollary~4.4]{K4}).
  \item [(ii)] A (continuous) theory has $NIP$ if and only if for each separable model $M$ and each formula $\phi$, the Banach space $V_{\phi,M}$ does not contain an isomorphic copy of $\ell_1$ (cf. \cite[Corollary~4.7]{K4}).
\end{itemize}
Our observation in the present paper (Proposition~\ref{SOP->c0}) shows that $SOP$ implies the existence of an isomorph of $c_0$ in some Banach space $V_{\phi,M}$, where $\phi$ is a formula and $M$ is a separable model. To summarize:

\begin{Proposition}
Let $T$ be a \textbf{continuous} theory. Then (i) implies (ii).
\begin{itemize}
  \item[(i)] (a) For some formula $\phi$ and some model $M$, the Banach space $V_{\phi,M}$ is nonreflexive, and (b) for each formula $\phi$ and each separable model $M$, the Banach space $V_{\phi,M}$ does not contain $\ell_1$ or $c_0$.
  \item[(ii)] $T$ is unstable but $NIP$ and $NSOP$.
\end{itemize}
\end{Proposition}

In case (i), for some formula $\phi$ and model $M$, $V_{\phi,M}$ is nonreflexive and does not contain an isomorph of $\ell_1$ or $c_0$. Recall that the first example of such a space was given in \cite{James} by Robert C. James. We conjecture that these types of spaces, and especially exotic Banach spaces, are the spaces whose theories are unstable but $NIP$ and $NSOP$. This conjecture is related to a question posed by Odell and Gowers on the existence of $c_0$ and $\ell_p$ ($1 \leq p < \infty$) in `explicitly defined' Banach spaces (see \cite[(Q1)]{K-Banach} and a discussion on {\em explicit definability} of norms therein). We believe that a {\em complete and enlightening} answer to the Odell--Gowers question (which is discussed in \cite{K-Banach}) is necessarily possible through the model-theoretic classification of continuous first-order theories.


\bigskip
Let us make more comparisons between the above theorems, namely Shelah's theorem for classical logic, the Eberlein--\v{S}mulian Theorem, Theorem~$\cal B$ above, and Theorems~A of \cite{HOR}. Roughly speaking, Theorem~A(a) of \cite[page~1]{HOR} asserts that a separable Banach space $X$ has an isomorphic copy of $\ell_1$ if and only if its bidual has a non-Baire-1 point. Theorem~A(b) asserts that $X$ has an isomorphic copy of $c_0$ if and only if its bidual has a non-continuous but $DBSC$ point.

Notice that Theorem~A and the Eberlein--\v{S}mulian Theorem study some properties inside a {\em model}, but Shelah's theorem and Theorem~$\cal B$ study properties of the {\em theory} of a model. Roughly (and possibly incorrectly speaking), assuming the compactness theorem of logic, Shelah's theorem and the Eberlein--\v{S}mulian Theorem are equivalent (cf. \cite[Section~3]{K-Baire}). Theorems~A and Theorem~$\cal B$ are not as perfect as the other two theorems, because of the gaps $DBSC \neq B_{1/4} \neq B_{1/2} \neq B_1$. Again, we point out that Shelah's theorem (for classical logic) works because $DBSC = B_{1/2}$ for simple functions, and the Eberlein--\v{S}mulian Theorem works because relative sequential compactness and sequential completeness imply relative compactness, equivalently $B_1 = B_1$ in our notation!

\subsection{Thesis}
The above results and the observations in \cite[Section~3]{K-Baire} suggest a new classification of continuous first-order theories. Indeed, let $X$ be a compact metric space and ${\mathcal C}(X)$ be some subclass of $B_1(X)$, containing $B_{1/2}(X)$. We say that a theory $T$ has the $\mathcal C$-property if

\medskip
\ \ \ ``For any formula $\phi(x,y)$ and any infinite sequence $(a_i : i < \omega)$, \textbf{if} the sequence $(\phi(a_i, y): S_{\tilde\phi}(\{a_i : i < \omega\}) \to [0,1])$ converges to a function in ${\mathcal C}(S_{\tilde\phi}(\{a_i : i < \omega\}))$, \textbf{then} there is no infinite sequence $(b_i : i < \omega)$ such that $(a_i),(b_i)$ witness the order property for $\phi(x,y)$."
\medskip

Note that the above definition is conditional, meaning that if ${\cal C}_1$ is a subclass of ${\cal C}_2$, then the ${\cal C}_1$-property is weaker than the ${\cal C}_2$-property.

Now we can provide several Shelah-like theorems.

\begin{Proposition} \label{Shelah-like2}
Let $T$ be a complete (continuous) theory and $\mathcal C$ as above. Then $T$ is stable if and only if it has both $NIP$ and the $\mathcal C$-property.
\end{Proposition}
\begin{proof}
Clearly, stability implies $NIP$. Also, by Grothendieck's criterion,  stability implies the $\mathcal C$-property, that is, the limit of every convergent sequence is continuous.  For the converse, by $NIP$, for any $\phi(x,y)$ and any $(a_i)$, there is a subsequence $(a_{i_j})$ such that $\phi(a_{i_j}, y)$ converges to a function $f$ which is $B_{1/2}$. As $B_{1/2} \subseteq \mathcal C$, by the $\mathcal C$-property, there is no infinite sequence $(b_i)$ such that $(a_{i_j}),(b_i)$ witness the order property for $\phi(x,y)$. By Grothendieck's criterion, $f$ is continuous. As $\phi$ and $(a_i)$ are arbitrary, by the Eberlein--\v{S}mulian Theorem and Grothendieck's criterion, $T$ is stable.
\end{proof}

Kechris and Louveau \cite{KL} studied various ordinal ranks of bounded Baire-1 functions. They introduced the classes of bounded Baire-1 functions ${\mathcal B}_1^\xi$ (for $\xi$ a countable ordinal). It is shown that (i) ${\mathcal B}_1^1=B_{1/2}$, (ii) ${\mathcal B}_1^\xi$ is a Banach subalgebra of Baire-1 functions (with the sup-norm), and (iii) ${\mathcal B}_1^\xi \subseteq {\mathcal B}_1^\zeta$ for countable ordinals $\xi \leq \zeta < \omega_1$. Therefore, by Proposition~\ref{Shelah-like2}, we have the following:

\begin{Corollary}
Let $T$ be a complete theory and $\xi$ a countable ordinal. Then $T$ is stable if and only if it has both $NIP$ and the ${\mathcal B}_1^\xi$-property.
\end{Corollary}

In Section~3 of \cite{KL}, there is a classification of functions between $DBSC$ and Baire-1/2. Recall from Fact~\ref{properties}(iv) that every $\{0,1\}$-valued Baire-1/2 function is $DBSC$. This explains why one can provide a better classification for {\em continuous} first-order theories that is not possible for the classical case.

\medskip
At the end of the paper, let us discuss what it means to have a dividing line in continuous logic and Banach space theory. In \cite[Chapter~2.3]{Baldwin-book}, Baldwin distinguishes between a virtuous property of a theory and a dividing line in classical logic. A property is {\em virtuous} if it has significant mathematical consequences for the theory or its models. A property is a {\em dividing line} if it and its negation are both virtuous. According to these definitions, by Theorems~\ref{NIP=Baire-1/2} and \ref{NIP definability}, since $NIP$ and its negation give us a lot of information about the theory/models, $NIP$ is a dividing line in continuous logic as well as in Banach space theory. Indeed, $NIP$ is equivalent to Baire-1/2 definability of coheirs, and its negation ($IP$) is equivalent to the existence of $\ell_1$ in the function space on types. In fact, such a dividing line already exists in Banach space theory, as seen in Rosenthal's $\ell_1$-theorem. These theorems provide the kind of control that was needed; that is, control on both sides of the dividing line. Another example of such a theorem in Banach space theory is Gowers' theorem \cite{Gowers2}: Every Banach space $X$ has a subspace $Y$ that either has an unconditional basis or is hereditarily indecomposable. As Gowers argues in \cite{Gowers1}: ``{\em the questions that were traditionally asked about nice subspaces [i.e., existence of an infinite-dimensional reflexive subspace or an isomorph of $\ell_1$ or $c_0$] were of the wrong kind (although more examples of Banach spaces were needed to make this clear).}" As $NIP$ corresponds to Rosenthal's $\ell_1$ theorem, several questions arise. Is there a model-theoretic property corresponding to Gowers' theorem? Are there model-theoretic dividing lines that separate {\em natural} and {\em exotic} Banach spaces? What are the connections between the known classification (in classical logic) and the classification presented in this article? These and many other questions are worth exploring.

Finally, the above points strongly inspire us to believe that Shelah's {\em dividing lines} and classification theory is the {\em correct} approach to many questions in various areas of mathematics.

\bigskip\noindent
 {\bf Acknowledgements.}
  I am very much indebted to Professor John T. Baldwin
for his kindness and his helpful comments.  I thank James Hanson for informing me that he has studied $SOP$-like concepts in continuous logic in his thesis. 
I would like to thank the anonymous referee for the valuable comments and corrections that have significantly improved the quality of this paper.

I would like to thank School of Mathematics,  Institute for Research in Fundamental Sciences  (IPM), P. O. Box 19395-5746, Tehran, Iran, for their support. This research was in part supported by a grant from IPM (No. 1403030025).

\vspace{10mm}


\end{document}